\pdfoutput=1

\documentclass[11pt]{amsart}
\usepackage{amsmath,amsthm,amssymb}
\usepackage{graphicx}
\usepackage[margin=1.1in]{geometry}

\begin{document}


\newcommand{\CH}{conv(G)}

\newcommand{\tubingset}[1]{\mathbb{T}_{#1}}

\newcommand{\HY} {{\mathcal H}(G)}
\newcommand{\C} {{\mathbb C}}                              
\newcommand{\R} {{\mathbb R}}                              

\newcommand{\US} {{s}}  
\newcommand{\LF} {{m}}  

\newcommand{\K} {\mathcal{K}}                           

\newcommand{\KG} {{\mathcal{K}} G}
\newcommand{\CG} {{\mathcal{C}} G}                
\newcommand{\Y} {H}
\newcommand{\la} {\langle}
\newcommand{\ra} {\rangle}

\newcommand{\basec}{\triangle}
\newcommand{\base}{\triangle \hspace{-.11in} \triangle}
\newcommand{\tbase}{\triangle \hspace{-.11in} \triangle^*}

\newcommand{\hide}[1]{}

\newcommand{\red} {r}
\newcommand{\ray} {\rho}
\newcommand{\per}   {\mathcal{P}}

\newcommand{\suchthat} {\:\: | \:\:}
\newcommand{\ore} {\ \ {\it or} \ \ }
\newcommand{\oand} {\ \ {\it and} \ \ }

\newcommand{\Mod}{\overline{\mathcal M}}     
\newcommand{\uMod}{{\mathcal M}}     
\newcommand{\oM} [1] {\ensuremath{{\mathcal M}_{0,#1}(\R)}}                 
\newcommand{\M} [1] {\ensuremath{{\overline{\mathcal M}}{_{0, #1}(\R)}}}    
\newcommand{\cM} [1] {\ensuremath{{\mathcal M}_{0,#1}}}                     
\newcommand{\CM} [1] {\ensuremath{{\overline{\mathcal M}}{_{0, #1}}}}       

%
%

\theoremstyle{plain}
\newtheorem{thm}{Theorem}
\newtheorem{prop}[thm]{Proposition}
\newtheorem{cor}[thm]{Corollary}
\newtheorem{lem}[thm]{Lemma}
\newtheorem{conj}[thm]{Conjecture}
\newtheorem*{thmhull}{Theorem 11}

\theoremstyle{definition}
\newtheorem*{defn}{Definition}
\newtheorem*{exmp}{Example}

\theoremstyle{remark}
\newtheorem*{rem}{Remark}
\newtheorem*{hnote}{Historical Note}
\newtheorem*{nota}{Notation}
\newtheorem*{ack}{Acknowledgments}
\numberwithin{equation}{section}


\title{Pseudograph associahedra}

\subjclass[2000]{Primary 52B11, Secondary 55P48, 18D50}

\author{Michael Carr}
\address{M.\ Carr: Brandeis University, Waltham, MA 02453}
\email{m.p.carr@gmail.com}

\author{Satyan L.\ Devadoss}
\address{S.\ Devadoss: Williams College, Williamstown, MA 01267}
\email{satyan.devadoss@williams.edu}

\author{Stefan Forcey}
\address{S.\ Forcey: Tennessee State University, Nashville, TN 37209}
\email{sforcey@tnstate.edu}

\begin{abstract}
Given a simple graph $G$, the graph associahedron $\KG$ is a simple polytope whose face poset is based on the connected subgraphs of $G$.   This paper defines and constructs graph associahedra in a general context, for pseudographs with loops and multiple edges, which are also allowed to be disconnected.  We then consider deformations of pseudograph associahedra as their underlying graphs are altered by edge contractions and edge deletions.
\end{abstract}

\keywords{pseudograph, associahedron, tubings}

\maketitle

\baselineskip=17pt

%
%
\section{Introduction}

Given a simple, connected graph $G$, the graph associahedron $\KG$ is a convex polytope whose face poset is based on the connected subgraphs of $G$ \cite{cd}.  For special examples of graphs, the graph associahedra become well-known, sometimes classical polytopes.  For instance, when $G$ is a path, a cycle, or a complete graph, $\KG$ results in the associahedron, cyclohedron, and permutohedron, respectively.  A geometric realization was given in \cite{dev2}.  Figure~\ref{f:kwexmp} shows $\KG$ when $G$ is a path and a cycle with three nodes, resulting in the 2D associahedron and cyclohedron.

\begin{figure}[h]
\includegraphics{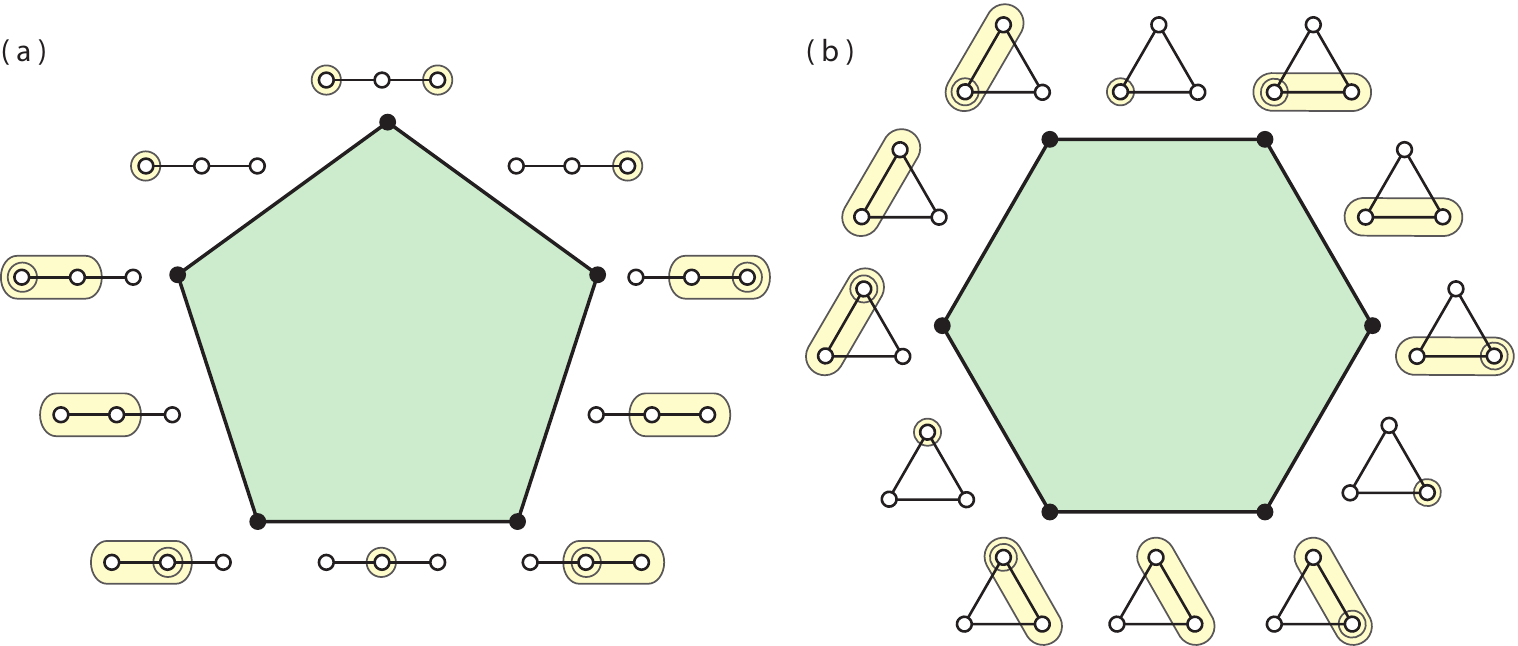}
\caption{Graph associahedra of the (a) path and (b) cycle with three nodes as underlying graphs.}
\label{f:kwexmp}
\end{figure}

This polytope was first motivated by De Concini and Procesi in their work on ``wonderful'' compactifications of hyperplane arrangements \cite{dp}.  In particular, if the hyperplane arrangement is associated to a Coxeter system, the graph associahedron $\KG$ appear as tilings of these spaces, where its underlying graph $G$ is the Coxeter graph of the system  \cite{djs}.  These compactified arrangements are themselves natural generalizations of the Deligne-Knudsen-Mumford compactification \M{n} of the real moduli space of curves \cite{dev1}.
From a combinatorics viewpoint, graph associahedra arise in relation to positive Bergman complexes of oriented matroids \cite{arw} along with studies of their enumerative properties \cite{prw}.  Recently, Bloom has shown graph associahedra arising in results between Seiberg-Witten Floer homology and Heegaard Floer homology \cite{blo}.  Most notably, these polytopes have emerged as graphical tests on ordinal data in biological statistics \cite{mps}.

It is not surprising to see $\KG$ in such a broad range of subjects. Indeed, the combinatorial and geometric structures of these polytopes capture and expose the fundamental concept of connectivity.  Thus far, however, $\KG$ have been studied for only simple graphs $G$.  
The goal of this paper is to define and construct graph associahedra in a general context: finite pseudographs which are allowed to be disconnected, with loops and multiple edges.  Most importantly, this induces a natural map between $\KG$ and $\KG'$, where $G$ and $G'$ are related by either edge contraction or edge deletion.  Such an operation is foundational, for instance, to the Tutte polynomial of a graph $G$, defined recursively using the graphs $G/e$ and $G-e$, which itself specializes to the Jones polynomial of knots. 

An overview of the paper is as follows:  Section~\ref{s:defns} supplies the definitions of the pseudograph associahedra along with several examples.  Section~\ref{s:construct} provides a construction of these polytopes and polytopal cones from iterated truncations of products of simplices and rays.   The connection to edge contractions (Section~\ref{s:contract}) and edge deletions (Section~\ref{s:delete}) are then presented.   A geometric realization is given in Section~\ref{s:real}, used to relate pseudographs with loops to those without.  Finally, proofs of the main theorems are given in Section~\ref{s:proof}.


\begin{ack}
The second author thanks Lior Pachter, Bernd Sturmfels, and the University of California at Berkeley for their hospitality during his 2009-2010 sabbatical where this work was finished.  
\end{ack}

%
%
\section{Definitions} \label{s:defns}
\subsection{}

We begin with foundational definitions.  Although graph associahedra were introduced and defined in \cite{cd}, we start here with a blank slate.  The reader is forewarned that definitions here might not exactly match those from earlier works since previous ones were designed to deal with just the case of simple graphs.

\begin{defn}
Let $G$ be a finite graph with connected components $G_1$, \ldots, $G_k$.
\begin{enumerate}
\item
A \emph{tube} is a proper connected subgraph of $G$ that includes at least one edge between every pair of nodes of $t$ if such edges of $G$ exist.
\item
Two tubes are \emph{compatible} if one properly contains the other, or if they are disjoint and cannot be connected by a single edge of $G$.  
\item
A \emph{tubing} of $G$ is a set of pairwise compatible tubes which cannot contain all of the tubes $G_1$, \ldots, $G_k$.
\end{enumerate}
\end{defn}


\begin{exmp}
The top row of Figure~\ref{f:tubings} shows examples of valid tubings, whereas the bottom row shows invalid ones.  Part (e) fails since one edge between the bottom two nodes must be in the tube.  The tubing in part (f) contains a non-proper tube of $G$.  The two tubes of part (g) fail to be compatible since they can be connected by a single edge of $G$.  And finally, the tubing of part (h) fails since it contains all the tubes of the connected components.
\end{exmp}

\begin{figure}[h]
\includegraphics{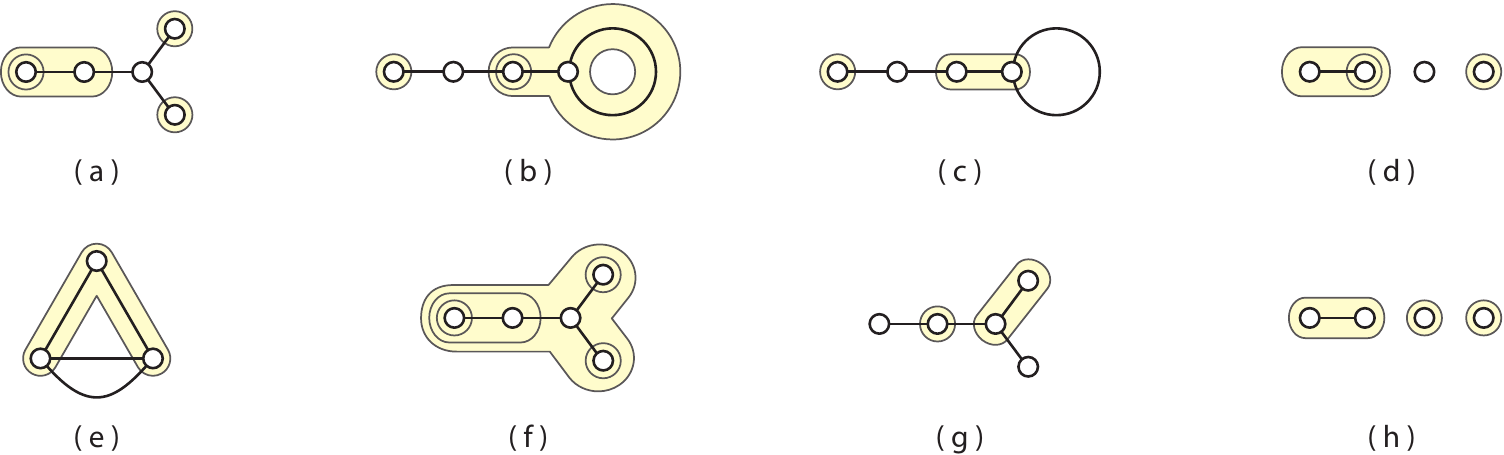}
\caption{The top row shows valid tubings and the bottom rows shows invalid ones.}
\label{f:tubings}
\end{figure}

\subsection{}

Let $\red$ be the number of \emph{redundant edges} of $G$, the minimal number of edges we can remove to get a simple graph.  We now state one of our main theorems.

\begin{thm} \label{t:pseudo}
Let $G$ be a finite graph with $n$ nodes and $\red$ redundant edges.  The \emph{pseudograph associahedron} $\KG$ is of dimension $n-1+\red$ and is either
\begin{enumerate}
\item a simple convex polytope when $G$ has no loops, \ or
\item a simple polytopal cone otherwise.
\end{enumerate}
Its face poset is isomorphic to the set of tubings of $G$, ordered under reverse subset containment.  In particular, the codimension $k$ faces are in bijection with tubings of $G$ containing $k$ tubes.
\end{thm}

\noindent 
The proof of this theorem follows from the construction of pseudograph associahedra from truncations of products of simplices and rays, given by Theorem~\ref{t:trunc}.  The following result allows us to only consider \emph{connected} graphs $G$:

\begin{thm} \label{t:disconnect}
Let $G$ be a disconnected pseduograph with connected components $G_1, G_2, \ldots, G_k$.  Then $\KG$ is isomorphic to $\KG_1 \times \KG_2 \times \cdots \times \KG_k \times \Delta_{k-1}.$
\end{thm}

\begin{proof}
Any tubing of $G$ can be described as:
\begin{enumerate}
\item a listing of tubings $T_1 \in \KG_1, \ T_2 \in \KG_2, \ \ldots, \ T_k \in \KG_k$, \ and
\item for each component $G_i$ either including or excluding the tube $T_i = G_i$. 
\end{enumerate}
The second part of this description is clearly isomorphic to a tubing of the edgeless graph $H_k$ on $k$ nodes.  But from \cite[Section 3]{dev2}, since $\K H_k$ is the simplex $\Delta_{k-1}$, we are done.
\end{proof}

We now pause to illustrate several examples.

\begin{exmp}
We begin with the 1D cases.  Figure~\ref{f:1D-exmp}(a) shows the pseudograph associahedron of a path with two nodes.  The polytope is an interval, seen as the classical 1D associahedron.  Here, the interior of the interval, the maximal element in the poset structure, is labeled with the graph with no tubes.  Part (b) of the figure shows $\KG$ as a ray when $G$ is a loop.  Note that we cannot have the entire loop as a tube since all tubes must be proper subgraphs.
\end{exmp}

\begin{figure}[h]
\includegraphics{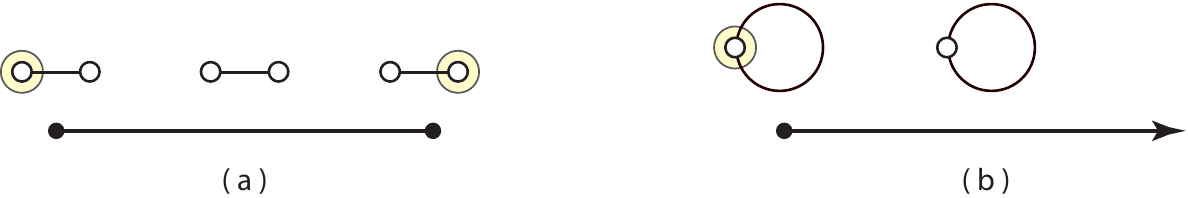}
\caption{Two 1D examples.}
\label{f:1D-exmp}
\end{figure}

\begin{exmp}
For some 2D cases, Figure~\ref{f:kwexmp} displays $\KG$ for a path and a cycle with three nodes as underlying graphs.  Figure~\ref{f:2D-exmp}(a) shows the simplest example of $\KG$ for a graph with a multiedge, resulting in a square.  The vertices of the square are labeled with tubings with two tubes, the edges with tubings with one tube, and the interior with no tubes.  Figure~\ref{f:2D-exmp}(b) shows $\KG$, for $G$ an edge with a loop, as a polygonal cone, with three vertices, two edges, and two rays. We will explore this figure below in further detail.
\end{exmp}

\begin{figure}[h]
\includegraphics{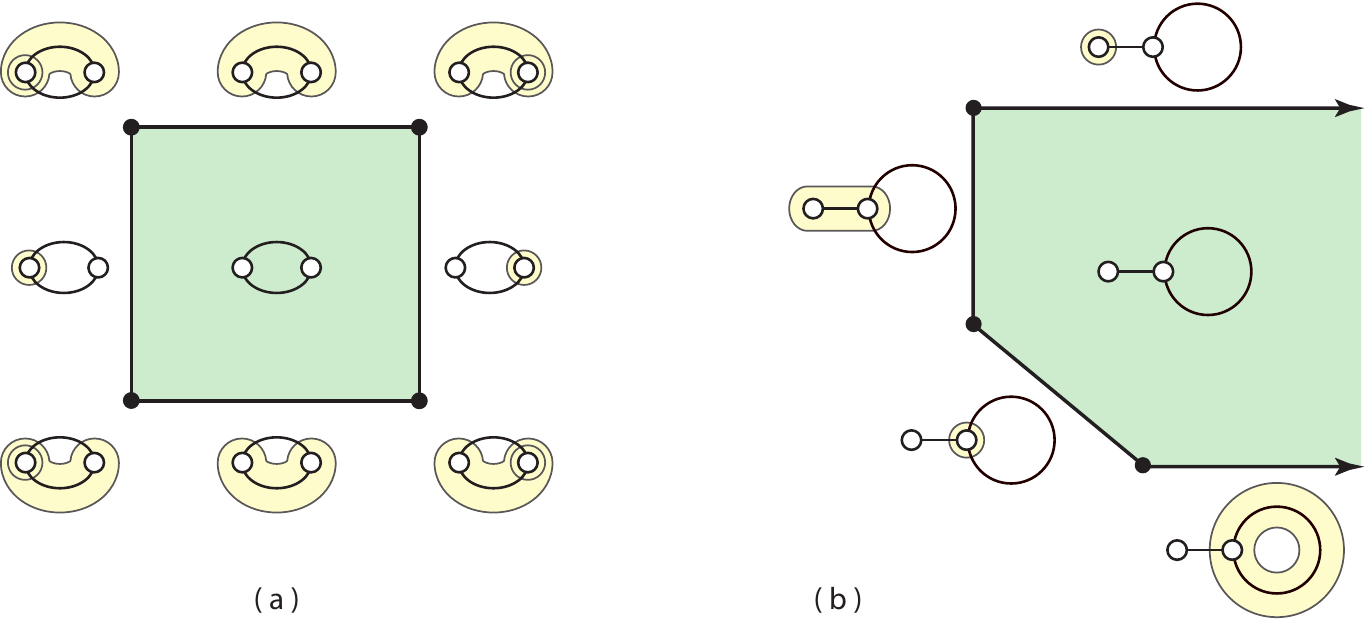}
\caption{Two 2D examples.}
\label{f:2D-exmp}
\end{figure}

\begin{exmp}
Three examples of 3D pseudograph associahedra are given in Figure~\ref{f:3D-exmp}.  Since each of the corresponding graphs have 3 nodes and one multiedge, the dimension of the polytope is three, as given in Theorem~\ref{t:pseudo}.  Theorem~\ref{t:disconnect} shows part (a) as the product of an interval (having two components) with the square from Figure~\ref{f:2D-exmp}(a), resulting in a cube.  The polyhedra in parts (b) and (c) can be obtained from iterated truncations of the triangular prism. Section~\ref{s:construct} brings these constructions to light.
\end{exmp}

\begin{figure}[h]
\includegraphics{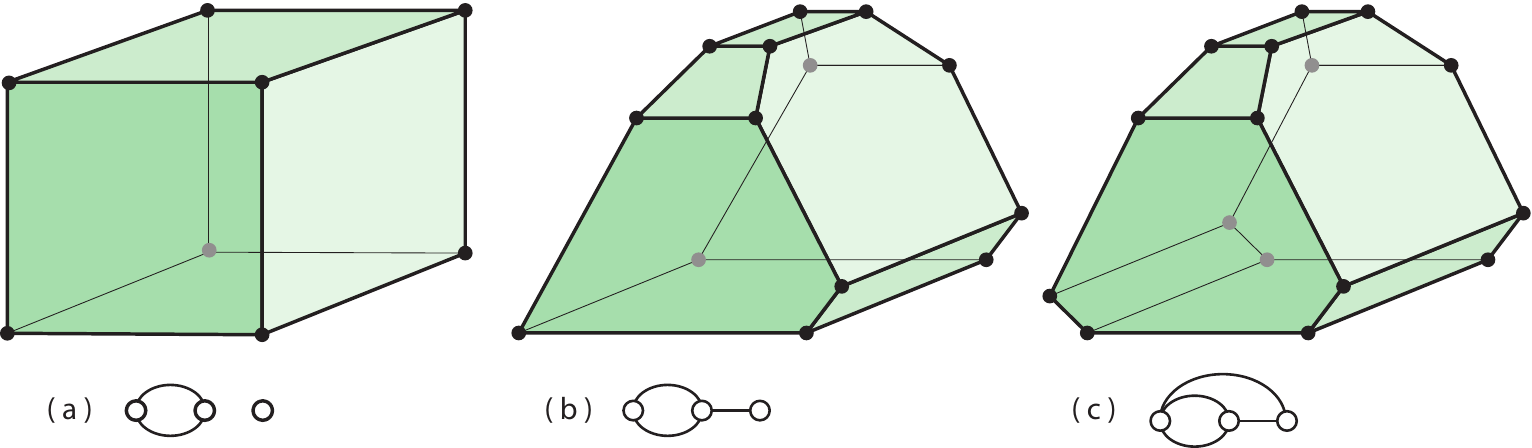}
\caption{Three 3D examples.}
\label{f:3D-exmp}
\end{figure}

\subsection{}

We close this section with an elegant relationship between permutohedra and two of the simplest forms of pseudographs. 

\begin{defn}
The \emph{permutohedron} $\per_n$ is an $(n-1)$-dimensional polytope whose faces are in bijection with the strict weak orderings on $n$ letters.  In particular, the $n!$ vertices of $\per_n$ correspond to all permutations of $n$ letters.
\end{defn}

\noindent
The two-dimensional permutohedron $\per_3$ is the hexagon and the polyhedron $\per_4$ is depicted in Figure~\ref{f:tonks-facet}(a).  It was shown in \cite[Section 3]{dev2} that if  $\Gamma_n$ is a complete graph of $n$ nodes, then $\K \Gamma_n$ becomes $\per_n$.


\begin{prop} \label{p:permuto}
Consider the simplest forms of pseudographs $G$:
\begin{enumerate}
\item If $G$ has two nodes and $n$ edges between them, then $\KG$ is isomorphic to $\per_n \times \Delta_1$.
\item If $G$ has one node and $n$ loops, then $\KG$ is isomorphic to $\per_n \times \ray$, where $\ray$ is a ray.
\end{enumerate}
\end{prop}

\begin{proof}
Consider case (1):  We view $\per_n$ as $\K\Gamma_n$ for the complete graph on $n$ nodes $\{v_1,\dots v_n\}$, and the interval $\Delta_1$ as $\K\Gamma_2$ for the complete graph on two nodes $\{b_1, b_2\}$.
Let the nodes of $G$ be $\{a_1, a_2\}$ and its edges $\{e_1, \ldots, e_n\}$.  
We construct an isomorphism $\KG \to \K\Gamma_n \times \K\Gamma_2$ where a tube $G_t$ of $G$ maps to the tube 
$(\psi_1(t), \psi_2(t))$, where $\psi_1(t)$ is the connected subgraph of $\Gamma_n$ induced by the node set $\{v_i \suchthat e_i \in G_t\}$, and $\psi_2(t)$ is the node $\{b_i \suchthat a_i = G_t \}$.
This proves the first result; the proof of case (2) is similar, replacing the two nodes of $G$ with one node.
\end{proof}

\begin{exmp}
Figure~\ref{f:3D-permuto-flat}(a) shows a hexagonal prism, viewed as $\per_3 \times \Delta_1$.  It is the pseudograph associahedron of the graph with two nodes and three connecting edges.  Part (b) shows a 2D projection of $\per_3 \times \ray$, the hexagonal cone of a graph with three loops.  Indeed, as we will see later, the removal of a hexagonal facet in (a) yields the object in (b).
\end{exmp}

\begin{figure}[h]
\includegraphics{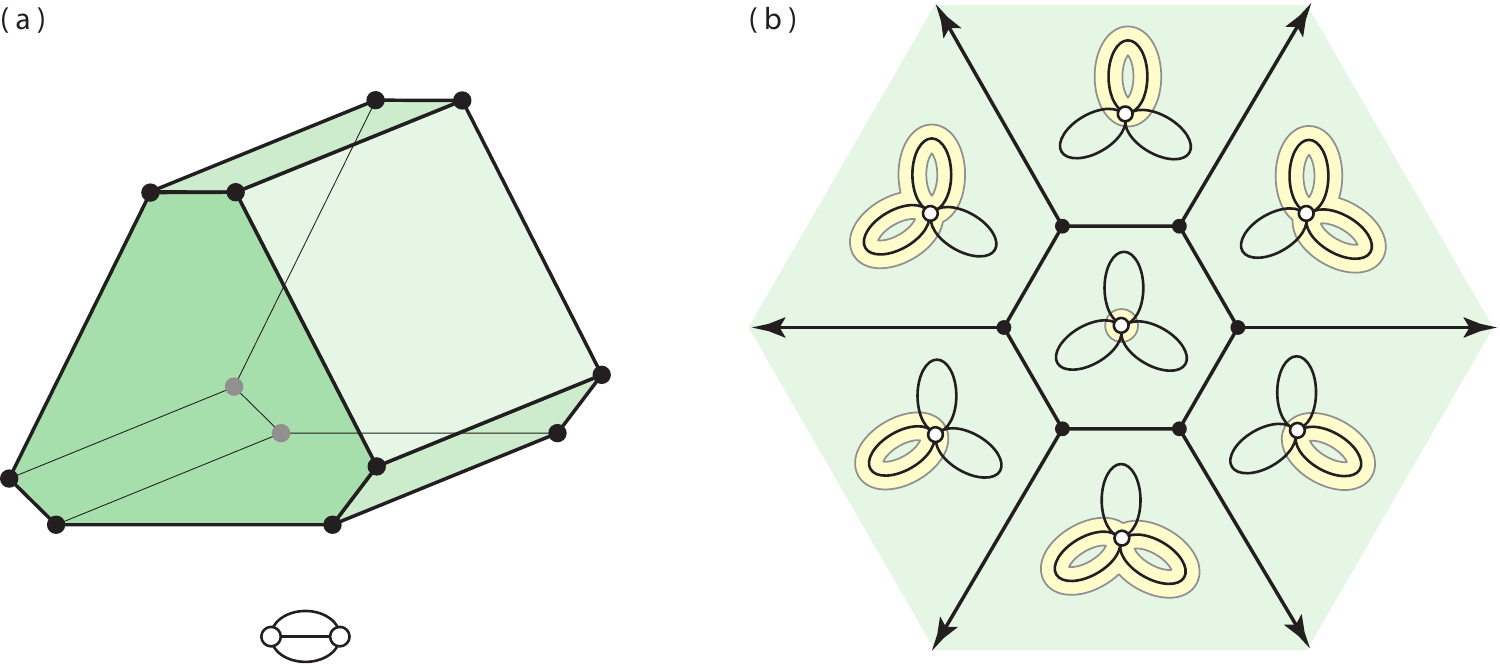}
\caption{(a) The hexagonal prism $\per_3 \times \Delta_1$ and (b) the planar projection of  $\per_3 \times \ray$.}
\label{f:3D-permuto-flat}
\end{figure}

%
%
\section{Constructions} \label{s:construct}
\subsection{}

There exists a natural construction of graph associahedra from iterated truncations of the simplex:  For a connected, simple graph $G$ with $n$ nodes, let $\basec_G$ be the $(n{-}1)$-simplex $\Delta_{n-1}$ in which each facet (codimension one face) corresponds to a particular node.  Thus each proper subset of nodes of $G$ corresponds to a unique face of $\basec_G$ defined by the intersection of the faces associated to those nodes.  Label each face of $\basec_G$ with the subgraph of $G$ induced by the subset of nodes associated to it.

\begin{thm} \cite[Section 2]{cd} \label{t:graphasstrunc}
For a connected, simple graph $G$, truncating faces of $\Delta_G$ labeled by tubes, in increasing order of dimension, results in the graph associahedron $\KG$.
\end{thm}

Figure~\ref{f:d4} provides an example of this construction.  It is worth noting two important features of this truncation.  First, only certain faces of the \emph{original} base simplex $\basec_G$ are truncated, not any new faces which appear after subsequent truncations.  And second, the \emph{order} in which the truncations are performed follow a De Concini - Procesi framework \cite{dp}, where all the dimension $k$ faces are truncated before truncating any $(k+1)$-dimensional faces.

\begin{figure}[h]
\includegraphics{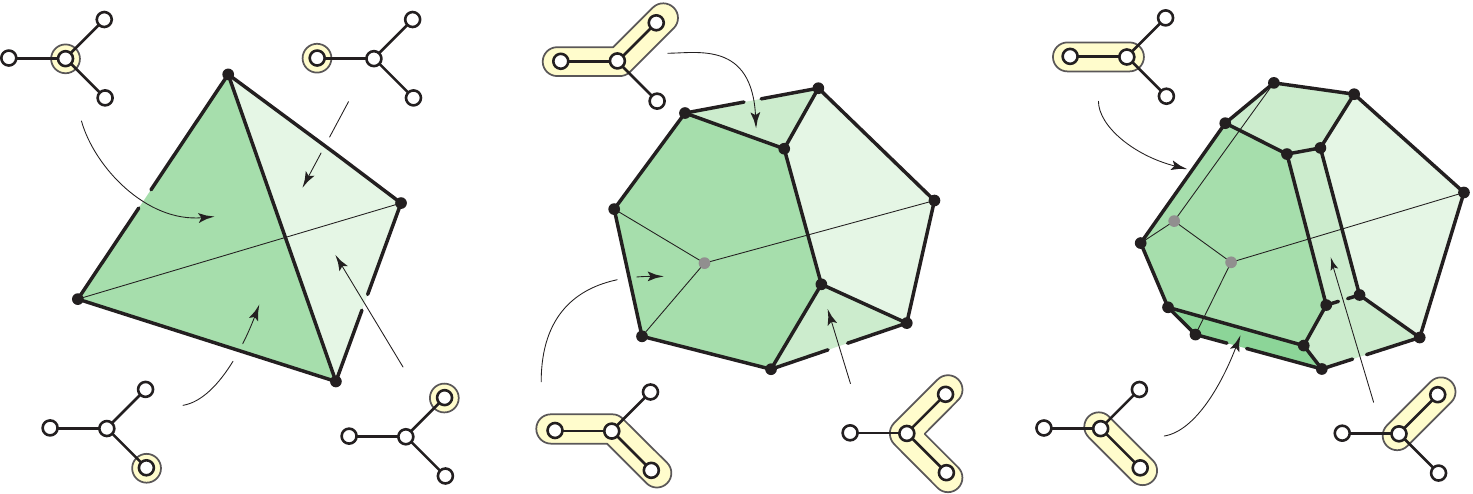}
\caption{An iterated truncation of the simplex resulting in a graph associahedron.}
\label{f:d4}
\end{figure}

\subsection{}

We construct the pseudograph associahedron by a similar series of truncations to a base polytope.  However the truncation procedure is a delicate one, where neither feature described above succeed here.

\begin{defn}
Let $G$ be a pseudograph with $n$ nodes.  Two (non-loop) edges of $G$ are in a \emph{bundle} if and only if they have the same pair of endpoints.  Let $G_\US$ be the \emph{underlying simple graph} of $G$, created by deleting all the loops and replacing each bundle with a single edge.\footnote{This graph is uniquely defined up to graph isomorphism.}  Figure~\ref{f:bundles}(a) shows an example of a pseudograph with 10 bundles and 4 loops, whereas part (b) shows its underlying simple graph.
\end{defn}

\begin{figure}[h]
\includegraphics{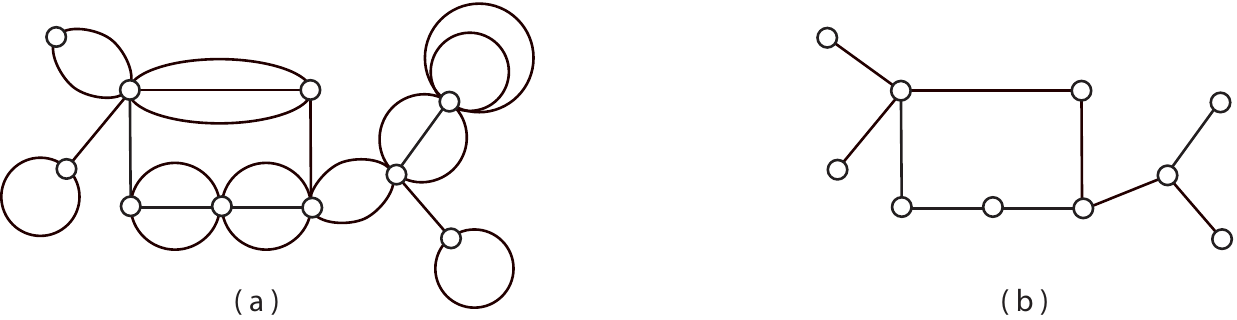}
\caption{(a) Pseudograph and (b) its underlying simple graph.}
\label{f:bundles}
\end{figure}

Let $\{B_1, \ldots, B_k\}$ be the set of bundles of edges of $G$, and denote $b_i$ as the number of edges of bundle $B_i$, and $\lambda$ as the number of loops of $G$.
Define $\base_G$ as the product 
$$\Delta_{n-1} \ \times \ \prod_{B_i \in G} \, \Delta_{b_i -1} \ \times \ \ray^{\lambda}$$
of simplices and rays endowed with the following labeling on its faces:
\begin{enumerate}
\item
Each \emph{facet} of the simplex $\Delta_{n-1}$ is labeled with a particular node of $G$, and each face of $\Delta_{n-1}$ corresponds to a proper subset of nodes of $G$, defined by the intersection of the facets associated to those nodes.
\item
Each \emph{vertex} of the simplex $\Delta_{b_i - 1}$ is labeled with a particular edge of bundle $B_i$, and each face of $\Delta_{b_i - 1}$ corresponds to a subset of edges of $B_i$ defined by the vertices spanning the face.
\item
Each \emph{ray} $\ray$ is labeled with a particular loop of $G$.
\item
These labelings naturally induce a labeling on $\base_G$.
\end{enumerate}

The construction of \emph{graph} associahedra from truncations of the simplex involved only a labeling associated to the nodes of our underlying graph.  Thus tubes of the graph are immediate, based on connected subgraphs containing certain nodes.  The construction of \emph{pseudograph} associahedra, however, involves the complexity of issues relating both the nodes and the edges.  This leads not only to a subtle choosing of the faces of $\base_G$ to truncate, but a delicate ordering of the truncation of the faces.

We begin by marking the faces of $\base_G$ which will be of interest in the truncation process:  To each tube $G_t$ of the labeled pseudograph $G$, associate a labeling $S$ of nodes and edges of $G$ such that
\begin{enumerate}
\item all nodes of $G_t$ are in $S$,
\item all edges of $G_t$ are in $S$, 
\item all bundles of $G$ not containing edges of $G_t$ are in $S$, \and
\item all loops not incident to any node of $G_t$ are in $S$.
\end{enumerate}

\begin{defn}
A tube $G_t$ is \emph{full} if it is a collection of bundles of $G$ which contains all the loops of $G$ incident to the nodes of $G_t$.  In other words, $G_t$ is an induced subgraph of $G$.
\end{defn}

\noindent
Figure~\ref{f:label} shows examples of tubes of a graph $G$ and their associated labeling $S$.  The two tubes on the top row are full, whereas the bottom four tubes are not.

\begin{figure}[h]
\includegraphics{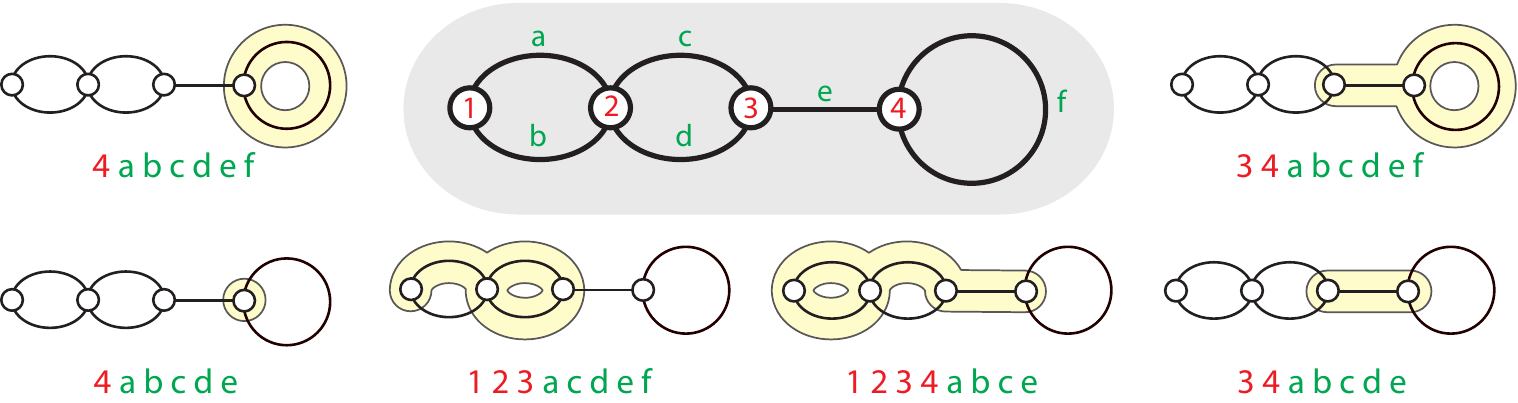}
\caption{Tubes and their corresponding labels in $\base_G$.}
\label{f:label}
\end{figure}

\subsection{}

We can now state our construction of $\KG$ from truncations, broken down into two steps:

\begin{lem} \label{l:simpletrunc}
Let $G$ be a connected pseudograph.   
Truncating the faces of $\base_G$ labeled with full tubes, in increasing order of dimension, constructs
\begin{equation} \label{e:midtrunc}
\K G_\US \ \times \ \prod_{B_i \in G} \triangle_{b_i-1}  \ \times \ \ray^{\lambda} \, .
\end{equation}
\end{lem}

\begin{proof}
A full tube consisting only of bundles maps to the $(b_i-1)$-face of $\Delta_{b_i-1}$.  Thus truncating these faces has a trivial effect on that portion of the product.  The result then follows immediately from Theorem~\ref{t:graphasstrunc}.
\end{proof}

As each face $f$ of $\base_G$ is truncated, those subfaces of $f$ that correspond to tubes but have not yet been truncated are removed.  
It is natural, however, to assign these defunct tubes to the combinatorial images of their original subfaces.  
Denote $\tbase_G$ as the truncated polytope of~\eqref{e:midtrunc}.

\begin{thm} \label{t:trunc}
Truncating the remaining faces of $\tbase_G$ labeled with tubes, in increasing order of the number of elements in each tube, results in the pseudograph associahedron $\KG$ polytope.
\end{thm}

\noindent
This immediately implies the combinatorial result of Theorem~\ref{t:pseudo}. The proof of this theorem is given in Section~\ref{s:proof}.  Notice the dimension of $\KG$ is the dimension of $\base_G$, which in turn equals $(n-1) + (b_i -1) + \cdots + (b_p -1) = n -1 +\red$, for $\red$ redundant edges, as claimed.

\begin{exmp}
We construct the pseudograph associahedron in Figure~\ref{f:3D-exmp}(b) from truncations.  The left side of Figure~\ref{f:3D-trunc-label} shows the pseudograph $G$ along with a labeling of its nodes and bundles. 
\begin{figure}[h]
\includegraphics{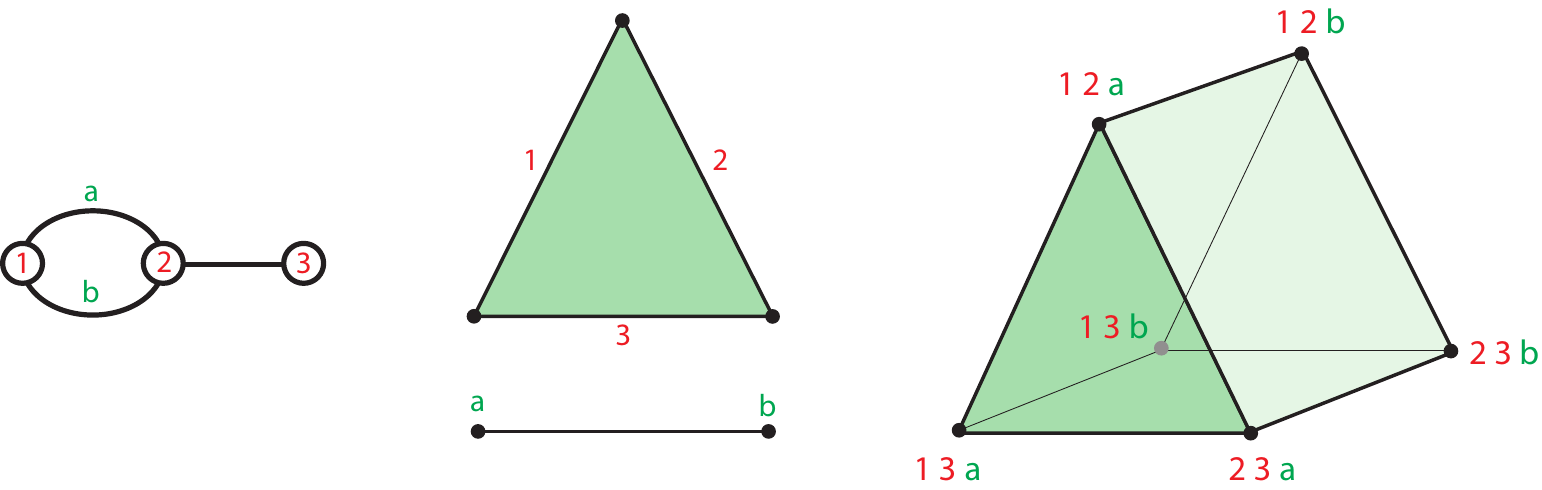}
\caption{A base polytope $\base_G$ and its labelings.}
\label{f:3D-trunc-label}
\end{figure}
(Notice the edge from node 2 to node 3 is not labeled since the bundle associated to this edge is the trivial $\Delta_0$ point.)  
Thus the base polytope $\base_G$ is the product of $\Delta_2 \times \Delta_1$, with the middle diagram providing the labeling on $\Delta_2$ and $\Delta_1$ from $G$.
The right side of the figure shows the induced labeling of the vertices of $\base_G$ from the labeling of $G$.

Figure~\ref{f:3D-trunc} shows the iterated truncation of $\base_G$ in order to arrive at $\KG$. 
\begin{figure}[h]
\includegraphics{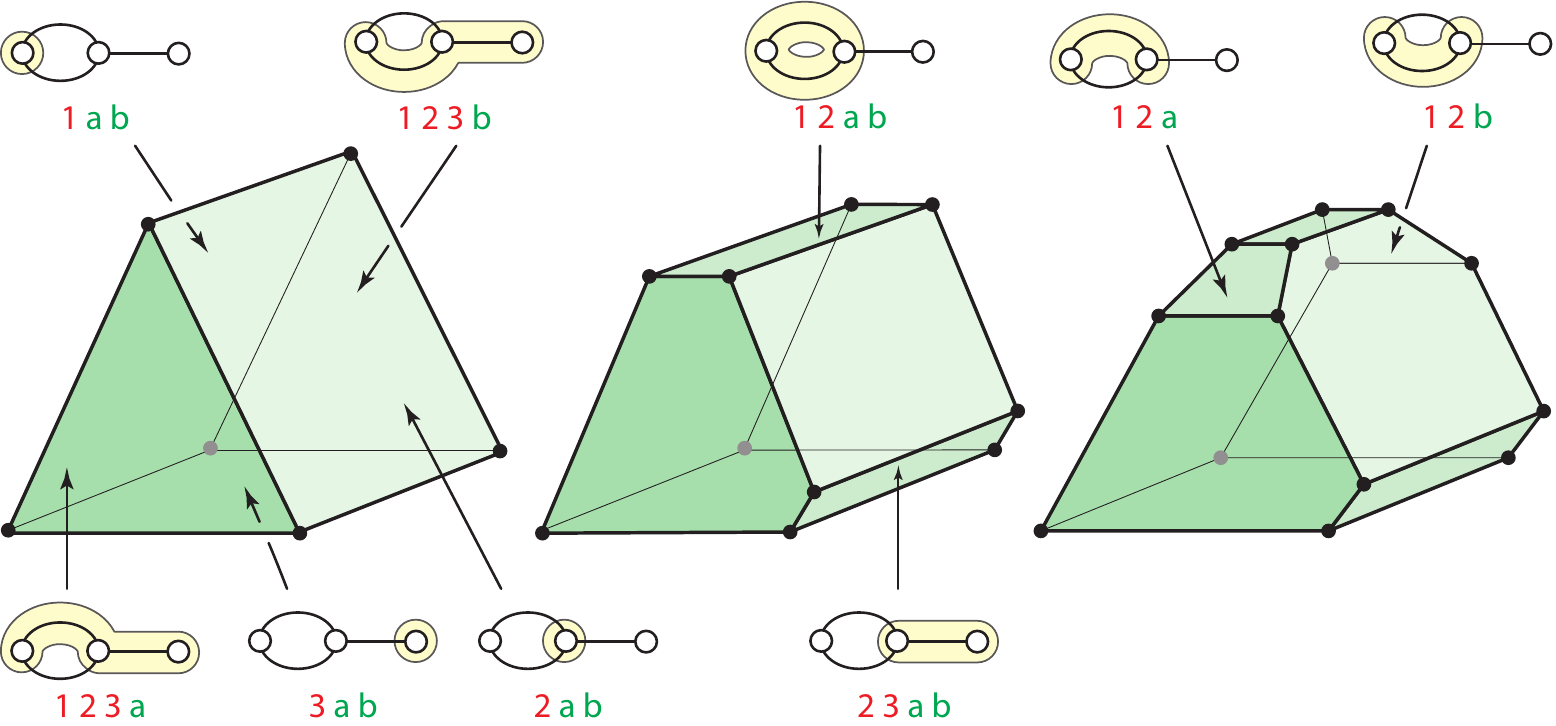}
\caption{Iterated truncations of $\base_G$ resulting in $\KG$ from  Figure~\ref{f:3D-exmp}(b).}
\label{f:3D-trunc}
\end{figure}
Lemma~\ref{l:simpletrunc} first requires truncating the faces of $\base_G$ labeled with full tubes.  There are five such faces in this case, three square facets and two edges.  Since the squares (labeled on the triangular prism on the left) are facets, their truncations do not change the topological structure of the resulting polyhedron.  The truncation of the two edges is given in the central picture of Figure~\ref{f:3D-trunc}, yielding $\tbase_G$.  This polytope is $\KG_\US \times \Delta_1$, a $K_4$ pentagonal prism, as guaranteed by the lemma.
Theorem~\ref{t:trunc} then requires truncations of the remaining faces labeled with tubes.
There are four such faces, two triangle facets (which are two facets of $\base_G$, labeled on the left of Figure~\ref{f:3D-trunc}) and two edges, resulting in the polyhedron $\KG$ on the right.  
\end{exmp}

\begin{exmp}
Let $G$ be a pseudograph of an edge with a loop attached at both nodes.   Figure~\ref{f:3D-loop} shows the polyhedral cone $\Delta_1 \times \ray^2$ along with the labeling of its four facets.
There are two full tubes, the front and back facets in (a), and thus their truncation does not alter the polyhedral cone.  
There are five other tubes to be truncated: two containing one element (a node), one with three elements (two nodes and an edge), and two facets with four elements (two nodes, one edge, one loop).  
By Theorem~\ref{t:trunc}, the truncation is performed in order of the number of elements in these tubes.
Figure~\ref{f:3D-loop}(b) shows the truncation of the edges assigned to tubes with one node.  Part (c) displays the result of truncating the edge labeled with a tube with three elements.  
\begin{figure}[h]
\includegraphics{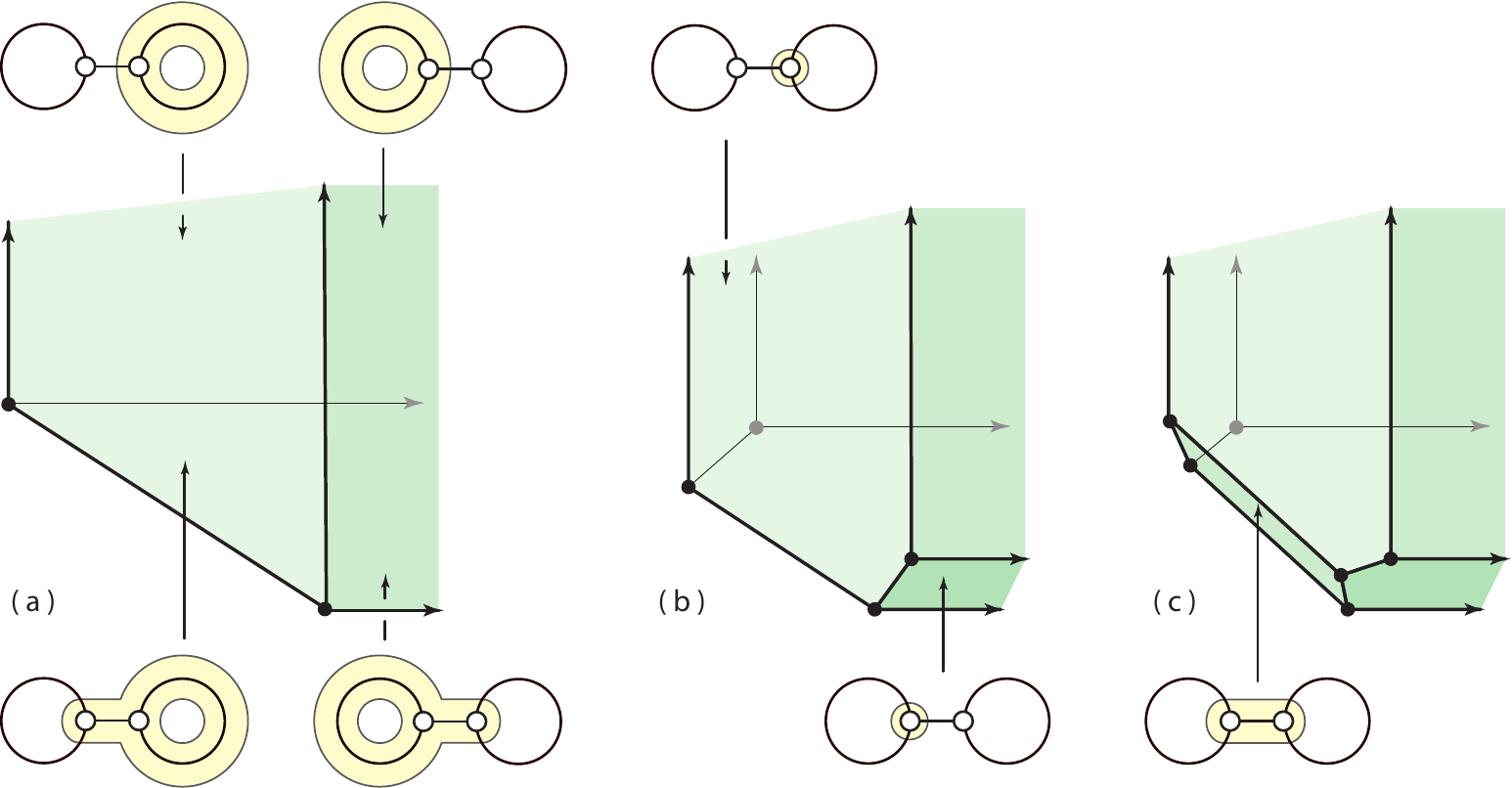}
\caption{An iterated truncation of $\Delta_1 \times \ray^2$, resulting in a graph associahedron.}
\label{f:3D-loop}
\end{figure}
\end{exmp}

\begin{exmp}
Figure~\ref{f:4D-ex1} displays a Schlegel diagram of the 4D tetrahedral prism $\Delta_3 \times \Delta_1$, viewed as the base polytope $\base_G$ of the pseudograph shown.
\begin{figure}[h]
\includegraphics{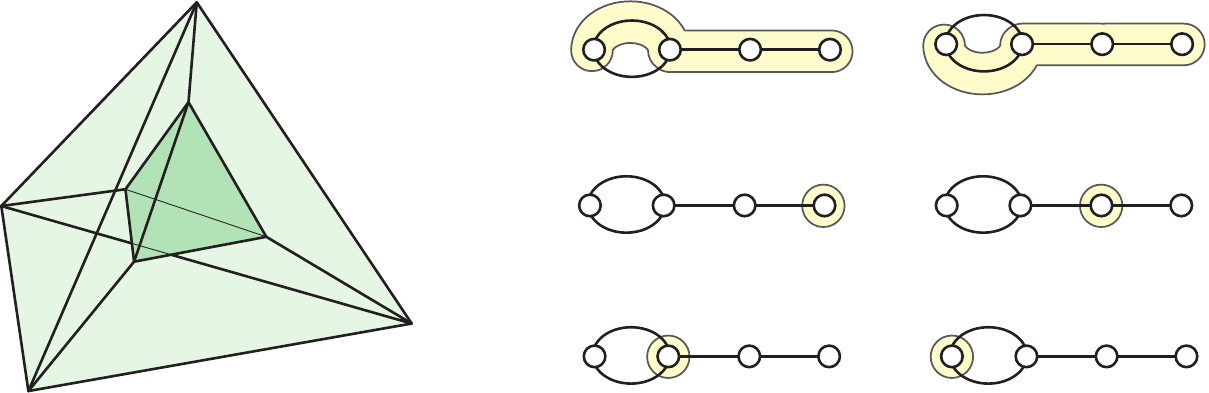}
\caption{A tetrahedral prism $\base_G$ along with labeling of tubes for its six facets.}
\label{f:4D-ex1}
\end{figure}
The six tubes of the pseudograph correspond to the six facets of $\base_G$.  The top two tubes are identified with tetrahedra whereas the other four are triangular prisms.  
Figure~\ref{f:4D-ex2} shows the iterated truncations of $\base_G$ needed to convert it into the pseudograph associahedron $\KG$.  
\begin{figure}[h]
\includegraphics{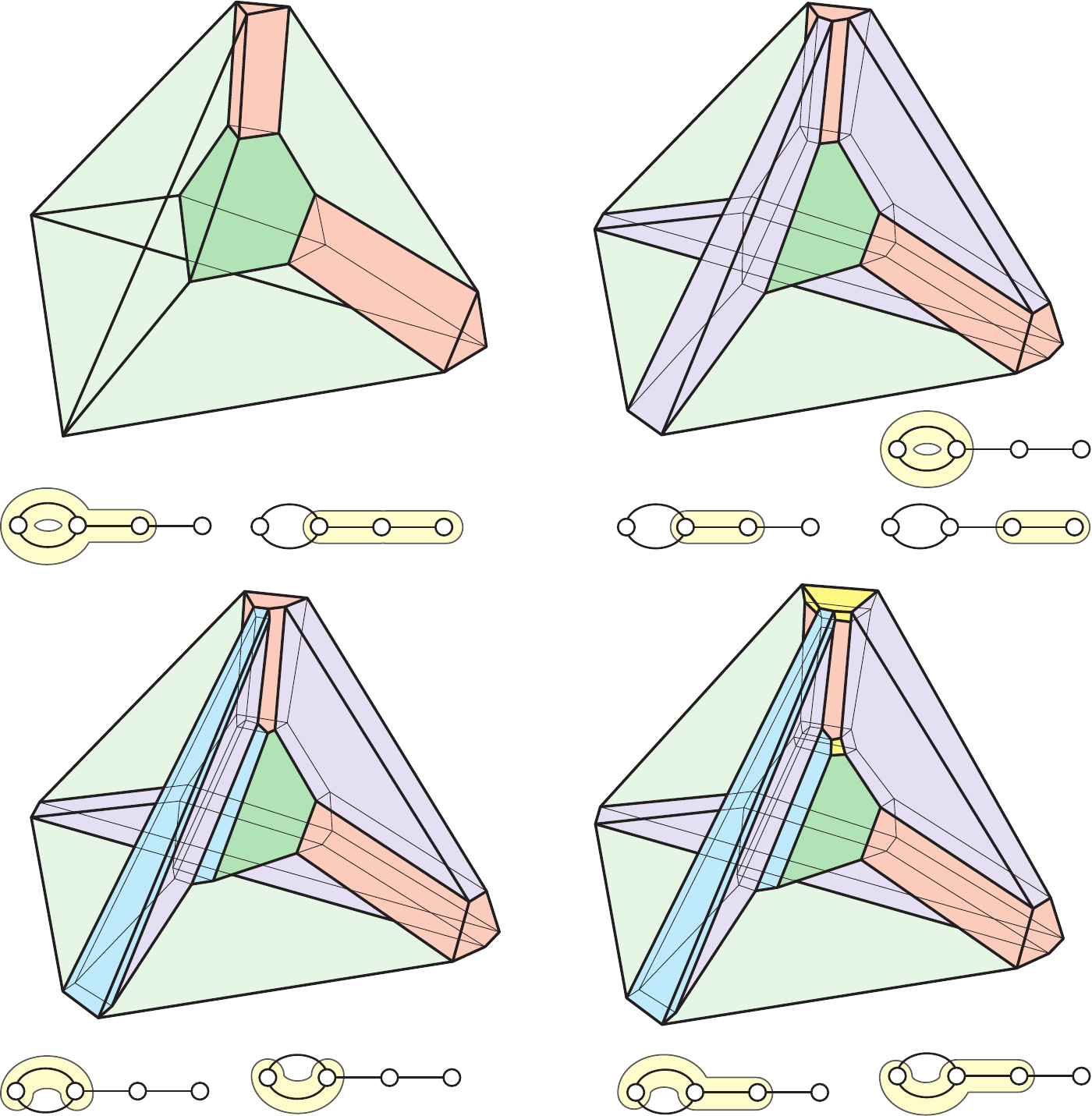}
\caption{An iterated truncation of the 4D tetrahedral prism, resulting in $\KG$.}
\label{f:4D-ex2}
\end{figure}
The first row shows two edges and three squares of $\base_G$ being truncated, which are labeled with full tubes.   The result, as promised by Lemma~\ref{l:simpletrunc} is $\KG_\US \times \Delta_1$, a $K_5$ associahedral prism.
We continue truncating as given by the bottom row, first two squares with three elements in their tubes, and then two pentagons, with five elements in their tubes. It is crucial that the truncations be performed in this order, resulting in $\KG$ as the bottom-right most picture.
\end{exmp}

%
%
\section{Edge Contractions}  \label{s:contract}

We have shown that any finite graph $G$ induces a polytope $\KG$.  Our interests now focus on deformations of pseudograph associahedra as their underlying graphs are altered.
 This section is concerned with contraction $G/e$ of an edge $e$, and the following section looks at edge deletions.   

\begin{defn}
An edge (loop) $e$ is \emph{excluded} by tube $G_t$ if $G_t$ contains the node(s) incident to $e$ but does not contain $e$ itself.
\end{defn}

\begin{defn}
Let $G$ be a pseudograph, $G_t$ a tube, and $e=(v,v')$ an edge.  Define
$$\Phi_e(G_t) =  
\begin{cases}
G_t & \ \ \ G_t \cap \{v,v'\} = \emptyset \\
G_t/e & \ \ \ e \in G_t \\
G_t/\{v,v'\} & \ \ \ \text{ $G_t$ excludes $e$ } \\
\emptyset  &  \ \ \ \text{ otherwise. }
\end{cases}$$
This map extends to $\Phi_e: \KG \to \K(G/e)$,
where given a tubing $T$ on $G$, $\Phi_e(T)$ is simply the set of tubes $\Phi_e(G_t)$ of $G/e$, for tubes $G_t$ in $T$.
\end{defn}

Figure~\ref{f:contract} shows examples of the map $\Phi_e$.  The top row displays some tubings on graphs where the edge $e$ to be contracted is highlighted in red.  The image of each tubing under $\Phi_e$ in $G/e$ is given below each graph.  Notice that $\Phi_e$ is not surjective in general since the dimension of $\K(G/e)$ can be arbitrarily higher than that of $\KG$.  For example, if $G$ is the complete bipartite graph $\Gamma_{2,n}$ with an extra edge $e$ between the two ``left'' nodes, then by Theorem~\ref{t:pseudo}, $\KG$ is of dimension $n+1$ whereas $\K(G/e)$ is of dimension $2n$.  Although not necessarily surjective, $\Phi_e$ is a poset map, as we now show.

\begin{figure}[h]
\includegraphics{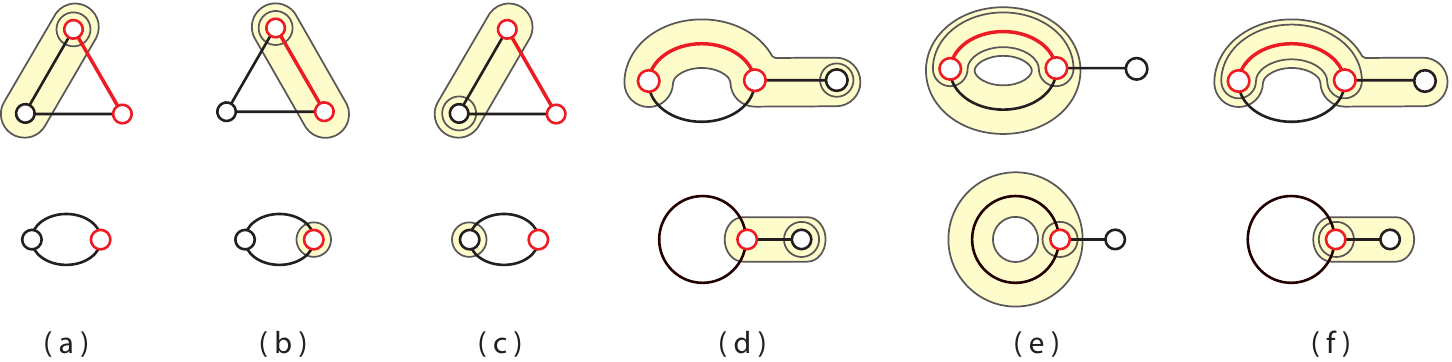}
\caption{Top row shows tubings on graphs, and the bottom row shows these tubings under the map $\Phi_e$, where the red edge $e$ has been contracted.}
\label{f:contract}
\end{figure}

\begin{prop} \label{p:contract}
For a pseudograph $G$ with edges $e$ and $e'$, $\Phi_e: \K G \to \K (G/e)$ is a poset map.  Moreover, the composition of these maps is commutative:  $\Phi_e \circ \Phi_{e'} = \Phi_{e'} \circ \Phi_{e}$.
\end{prop}

\begin{proof}
For two tubings $T$ and $T'$ of $G$, assume $T \prec T'$.  For any tube $G_t \in T'$,  the tube $\Phi_e(G_t)$ is included in both $\Phi_e(T)$ and $\Phi_e(T')$.  Thus $\Phi_e(T) \prec \Phi_e(T')$, preserving the face poset structure.   To check commutativity, it is straightforward to consider the 16 possible relationships of edges $e$ and $e'$ with a given tube $G_t$ of $G$, four each as in the definition of $\Phi_e(G_t)$.  For each possibility,
the actions of $\Phi_e$ and $\Phi_{e'}$ commute.
\end{proof}

For any collection $E$ of edges of $G$, let $\Phi_E:\KG \to \K(G/E)$ denote the composition of maps $\{\Phi_e ~|~e \in E\}.$  If $E$ is the set of edges of a connected subgraph $H$ of $G$, then contracting $E$ will collapse $H$ to a single node.  The resulting graph $G/H$ is the \emph{contraction} of $G$ with respect to $H$.   The following result describes the combinatorics of the facets of $\KG$ based on contraction.

\begin{thm} \label{t:prod}
Let $G$ be a connected pseudograph.  The facet associated to tube $G_t$ in $\KG$ is 
$$\KG_t \times M$$ 
where $M$ is the facet of $\K (G/G_t)$ associated to the single node of $G/G_t$ which $G_t$ collapses to.
In other words, the contraction map $\Phi_E: \KG \to \K(G/G_t)$ restricted to tubings of $G_t$ is the canonical projection onto $M$.
\end{thm}

\begin{proof} 
Let $v$ be the single node of $G/G_t$ which $G_t$ collapses to.
Given a tubing $T$ of the subgraph induced by $G_t$, and $T'$ a tubing of $G/G_t$ which contains the tube $\{v\}$, we define a map:
$$\rho(T,T') \ = \ T \ \cup \ \{G_t\} \ \cup \ \{G_{t'} \in T' ~|~ v \notin G_{t'} \} \ \cup \ \{ (G_{t'}-v) \cup G_t ~|~ v \in G_{t'} \in T'\} \, .$$
This is an isomorphism from the Cartesian product to the facet of $\KG$ corresponding to the tube $G_t$, which can be checked to preserve the poset structure.
\end{proof}

The following corollary describes the relationship between a graph and its underlying simple graph, at the level of graph associahedra.  

\begin{cor}
Let $G_\US$ be the underlying simple graph of a connected graph $G$ with $\red$ redundant edges.  The corresponding facet of $\KG$ for the tube $G_\US$ is equivalent to $\KG_\US \times \per_\red$.
\end{cor}

\begin{proof}
This follows immediately from Proposition~\ref{p:permuto} and Theorem~\ref{t:prod} above, since contracting the underlying simple graph $G_\US$ in $G$ gives us a bouquet of $n$ loops.
\end{proof}

\begin{exmp}
Figure~\ref{f:bouquet}(a) shows a graph $G$ with two nodes and seven edges, with one such edge $e$ highlighted in red.  By Proposition~\ref{p:permuto}, we know the pseudograph associahedron $\KG$ is the permutohedral prism $\per_7 \times \Delta_1$.  The tube given in part (b), again by Proposition~\ref{p:permuto}, is the permutohedron $\per_6$.  By the corollary above, we see $\per_6$ appearing as a codimension two face of $\per_7 \times \Delta_1$.  Figure~\ref{f:bouquet}(c) shows a graph $G$ and its underlying simple graph $G_\US$, outlined in red, and redrawn in (d).  The corresponding facet of tube $G_\US$ in $G$ is $\per_6$, the pseudograph associahedron of (b), and the pseudograph associahedron $\KG_\US$ of (d).
\end{exmp}

\begin{figure}[h]
\includegraphics{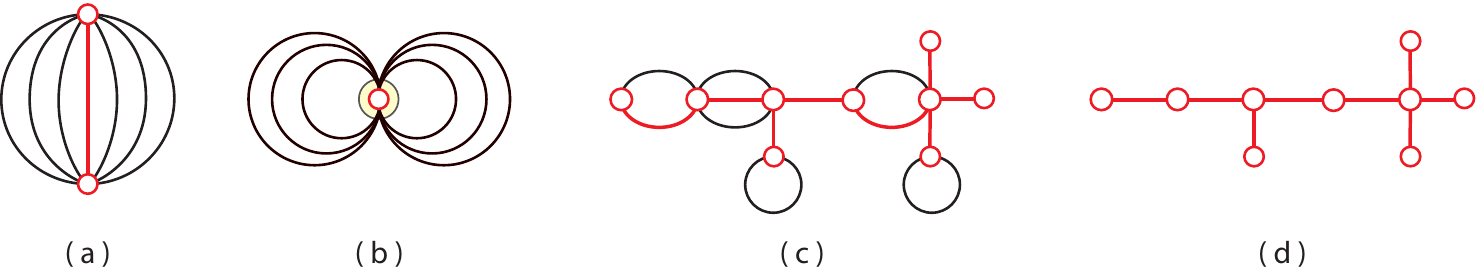}
\caption{Relationships between permutohedra and underlying graphs.}
\label{f:bouquet}
\end{figure}

%
%
\section{Edge Deletions}  \label{s:delete}
\subsection{}

We now turn our focus from edge contractions $G/e$ to edge deletions $G-e$.  Due to Theorem~\ref{t:disconnect}, we have had the luxury of assuming all our graphs to be connected, knowing that pseudograph associahedra for disconnected graphs is a trivial extension.  In this section, due to deletions of edges, no assumptions are placed on the graphs.

\begin{defn}
A \emph{cellular surjection} from  polytopes $P$ to $Q$ is a map $f$ from the face posets of $P$ to $Q$  which preserves the poset structure, and which is onto. That is, if $x$ is a subface of $y$ in
$P$ then $f(x)$ is a subface of or equal to $f(y)$.   It is a \emph{cellular projection} if it also has the property that the dimension of $f(x)$ is less than or equal to the dimension of $x$.
\end{defn}

In \cite{to}, Tonks found a cellular projection from permutohedron to associahedron. In this projection, a face of the permutohedron, represented by a \emph{leveled tree}, is taken to its underlying tree, which corresponds to a face of the associahedron.  The new revelation of Loday and Ronco \cite{lr} is that this map gives rise to a Hopf algebraic projection, where this algebra of binary trees is seen to be embedded in the Malvenuto-Reutenauer algebra of permutations.  Recent work by Forcey and Springfield \cite{fs} show a fine factorization of the Tonks cellular projection through all connected graph associahedra, and then an extension of the projection to disconnected graphs. Several of these cellular projections through polytopes are also shown to be algebra  and coalgebra homomorphisms. Here we further extend the maps based on deletion of edges to all pseudographs, in anticipation of future usefulness to both geometric and algebraic applications.

\begin{defn}
Let $G_t$ be a tube of $G$, where $e$ is an edge of $G_t$.  We say $e$ \emph{splits} $G_t$ into tubes $G_{t'}$ and $G_{t''}$ if $G_t - e$ results in two disconnected tubes $G_{t'}$ and $G_{t''}$ such that 
$$G_t \ = \ G_{t'} \cup G_{t''} \cup \{e\}.$$
\end{defn}

\begin{defn}
Let $G$ be a pseudograph, $G_t$ a tube and $e$ be an edge of $G$. Define 
$$\Theta_e(G_t) =  
\begin{cases}
G_t & \ \ \ \text{ if $e \notin G_t$ } \\
G_t-e & \ \ \ \text{ if $e \in G_t$ and $e$ does not split $G_t$ } \\
\{G_{t'}, G_{t''}\} & \ \ \ \text{ if $e$ splits $G_t$ into compatible tubes $G_{t'}$ and $G_{t''}$ } \\
\emptyset & \ \ \ \text{ otherwise.}
\end{cases}$$
This map extends to $\Theta_e: \KG \to \K(G - e)$,
where given a tubing $T$ on $G$, $\Theta_e(T)$ is simply the set of tubes $\Theta_e(G_t)$ of $G-e$, for tubes $G_t$ in $T$.
\end{defn}

Roughly, as a single edge is deleted, the tubing under $\Theta$ is preserved ``up to connection.'' That is, if the nodes of a tube $G_t$ are no longer connected by edge deletion, $\Theta(G_t)$ becomes the two tubes split by $e$, as long as these two tubes are compatible.  Figure~\ref{f:tonks-vertex} shows 
\begin{figure}[h]
\includegraphics{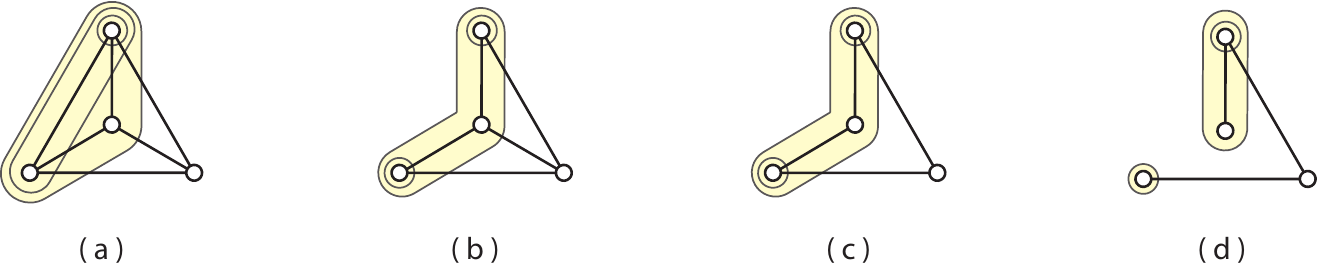}
\caption{The projection $\Theta$ factored by graphs, from the complete graph to the path.}
\label{f:tonks-vertex}
\end{figure}
maximal tubes on four different graphs, each corresponding to a vertex of its respective graph associahedron.  As an edge gets deleted from a graph, progressing to the next, the map $\Theta$ shows how the tubing is being factored through.  In this particular case, a vertex of the permutohedron (a) is factored through to a vertex of the associahedron (d) through two intermediary graph associahedra.  

\begin{rem}  
For a tubing $T$ of $G$ and a \emph{loop} $e$ of $G$, we find that the contraction and deletion maps of $e$ agree; that is, $\Theta_e(T) = \Phi_e(T)$.  
\end{rem}

\subsection{}

We now prove that $\Theta$ is indeed a cellular surjection, as desired.  The following is the analog of Proposition~\ref{p:contract} for edge deletions.  

\begin{prop} \label{p:delete}
For a pseudograph $G$ with edges $e$ and $e'$, $\Theta_e: \K G \to \K (G - e)$ is a cellular surjection. Moreover, the composition of these maps is commutative: $\Theta_e \circ \Theta_{e'} = \Theta_{e'} \circ \Theta_{e}$.
\end{prop}

\begin{proof}
For two tubings $U$ and $U'$ of $G$, assume $U \prec U'$.  For any tube $G_t \in U'$, the tube $\Theta_e(G_t)$ is included in both $\Theta_e(U)$ and $\Theta_e(U')$.  Thus $\Theta_e(U) \prec \Theta_e(U')$, preserving the face poset structure.

The map $\Theta$ is surjective, since given any tubing $U$ on $G-e$, we can find a preimage $T$ such that $U =
\Theta_e(T)$ as follows: First consider all the tubes of $U$ as a candidate tubing of $G$. If it is a valid
tubing, we have our $T.$ If not, there must be a pair of tubes $G_t'$ and $G_t''$ in $U$ which are adjacent via
the edge $e$ and for which there are no tubes containing either $G_t'$ or $G_t''$.  Let $U_1$ be the result of replacing that pair in $U$ with the single tube $G_t = G_t' \cup G_t''$. If $U_1$ is a valid tubing of $G$, then let $T=U_1$. If not, continue inductively.

To prove commutativity of map composition, consider the image of a tubing  of $G$ under either composition.  A tube of $G$ that is a tube of both $G-e$ and $G-e'$ will persist in the image. Otherwise it will be split into compatible tubes, perhaps twice, or forgotten. The same smaller tubes will result regardless of the order of the splitting. 
\end{proof}

\begin{rem}
If $e$ is the only edge between two nodes of $G$, then $\Theta_e$ will be a cellular \emph{projection} between two  polytopes or cones of the same dimension. Faces will only be mapped to faces of smaller or equal dimension. However, if $e$ is a multiedge, then $G-e$ is a tube of $G$.  In this case, the map $\Theta_e$ projects all of $\KG$ onto a single facet of $\KG$, where there may be faces mapped to a face of larger dimension. 
An example of a deleted multiedge is given in Figure~\ref{f:tonks-multi}.
\end{rem}

\begin{figure}[h]
\includegraphics{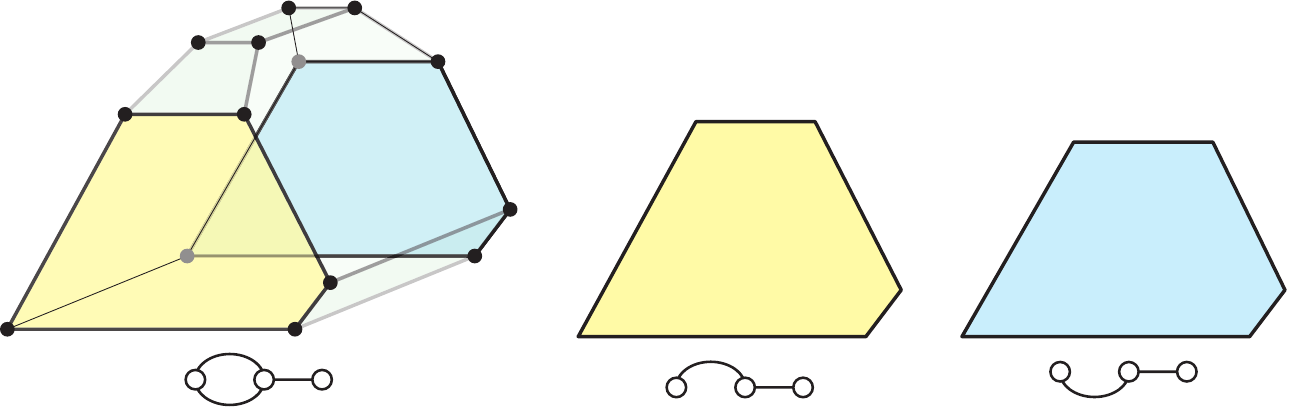}
\caption{The cellular surjection $\Theta_e$ for two different multiedges of $G$ from the example in Figure~\ref{f:3D-exmp}(b).}
\label{f:tonks-multi}
\end{figure}

For any collection $E$ of edges of $G$, denote $\Theta_E$
as the composition of projections $\{\Theta_e ~|~e \in E\}$.
Let $\Gamma_n$ be the complete graph on $n$ numbered nodes, and let $E$ be the set of all edges of $\Gamma_n$ except for the path in consecutive order from nodes $1$ to $n$.  Then $\Theta_E$ is equivalent to the Tonks projection \cite{fs}. Thus, by choosing any order of the edges to be deleted, there is a factorization of the Tonks cellular projection through various graph associahedra. An example of this, from the vertex perspective, was shown in Figure~\ref{f:tonks-vertex}.  

\begin{figure}[h]
\includegraphics{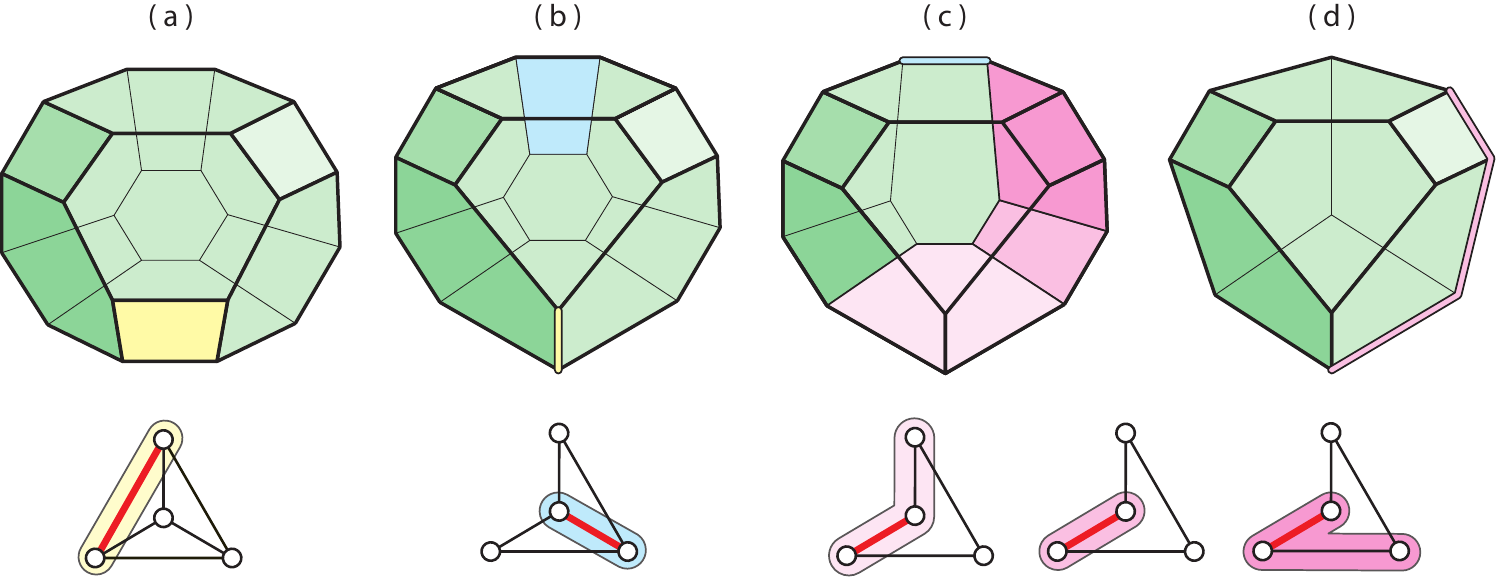}
\caption{A  factorization of the Tonks projection through 3D graph associahedra. The shaded facets
correspond to the shown tubings, and are collapsed as indicated to respective edges.   The permutohedron $P_4$ in (a), through a sequence of collapses, is transformed to the associahedron $K_5$ in (d).} 
\label{f:tonks-facet}
\end{figure}

The same map, from the facet viewpoint, is given in Figure~\ref{f:tonks-facet}.  Part (a) shows the permutohedron $\per_4$, viewed as $\K\Gamma_4$.  A facet of this polyhedron is highlighted and below it is the tube associated to the facet.  Deleting the (red) edge in the tube, thereby splitting the tube into two tubes, corresponds to collapsing the quadrilateral face into an interval, shown in part (b).  A similar process is outlined going from (b) to (c).  Figure~\ref{f:tonks-facet}(c) shows three faces which are highlighted, each with a corresponding tube depicted below the polyhedron.   These are the three possible tubes such that deleting the (red) edge of each tube produces a splitting of the tube into two compatible tubes.  Such a split corresponds to the collapse of the three marked facets of (c), resulting in the associahedron shown in (d).


%
%
\section{Realization}   \label{s:real}
\subsection{}

Let $G$ be a pseudograph without loops.  We now present a realization of $\KG$, assigning an integer coordinate to each of its vertices.  From Theorem~\ref{t:pseudo}, the vertices of $\KG$ are in bijection with the maximal tubings of $G$.  For each such maximal tubing $T$, we first define a map $f_T$ on each edge of each \emph{bundle} of $G$.   

\begin{nota}
Let $|G|$ denote the number of nodes and edges of $G$.  For a tube $G_t$, let $V(t)$ denote the node set of $G_t$, and let $E(i,t)$ denote the edges of bundle $B_i$ in $G_t$.
\end{nota}

For a given tubing $T$, order the edges of each bundle $B_i$ by the number of tubes of $T$ that do \emph{not} contain each $e$ in $B_i$.  Let $e(i,j)$ refer to the $j$-th edge in bundle $B_i$ under this ordering.  Thus $e(i,j)$ is contained in more tubes than $e(i,j+1)$.  Let $G_{e(i,j)}$ be the largest tube in $T$ that contains $e(i,j)$ but not $e(i,j+1)$.  Note that $G_{e(i,b_i)}$ is the entire graph $G$.  We assign a value $f_T$ to each edge in each bundle of $G$, as follows:
$$f_T(e(i,j))  \ = \  
\begin{cases}
\ \ \displaystyle{c \ + \ \sum_{x=1}^{b_i-1} \bigg(\ 2 \left|G - G_{e(i,x)}\right|  \ - \ 1 \ \bigg)} & \ \ \ \ j=1 \\[.3in]
\ \ \displaystyle{c^{\; j-1} \cdot (c-1) \ - \ \bigg(\ 2 \left|G - G_{e(i,j-1)}\right|  \ - \ 1 \ \bigg)}  & \ \ \ \ j \neq 1 \\
\end{cases}$$
for the constant $c = |G|^2$.
We assign $f_T(v)$ to each node of $G$ recursively by visiting each tube of $T$ in increasing order of size and ensuring that for all nodes and edges $x \in G_t$,
$$\sum_{x \in G_t} f_T(x) \ = \ c^{\; |V(t)|} \ + \ \sum_i c^{\; E(i,t)} \ + \ \left|G - G_t\right|^2 \, .$$

\begin{thm} \label{t:hull}
Let $G$ be a pseudograph without loops, with an ordering $v_1, v_2, \ldots, v_n$ of its nodes, and an ordering $e_1, e_2, \ldots, e_k$ of its edges.  For each maximal tubing $T$ of $G$, the convex hull of the points
\begin{equation}
\label{e:hull}
\bigg(f_T(v_1), \ldots, f_T(v_n), f_T(e_1), \ldots, f_T(e_k)\bigg)
\end{equation}
in $\R^{n+k}$  yields the pseudograph associahedron $\KG$.
\end{thm}

\noindent The proof of this is given at the end of the paper.

\subsection{}

We now extend the realization above to pseudographs with loops.  In particular, we show every pseudograph associahedra with loops can be reinterpreted as an open subcomplex of one without loops, via a subtle redescription of the loops.

\begin{defn}
For $G$ a connected pseudograph with loops, define an associated \emph{loop-free} pseudograph $G_{\LF}$ by replacing the set of loops attached to a node $v$ by a set of edges between $v$ and a new node $v'$. We call $v'$ a \emph{ghost node} of $G_\LF$. An example is given in Figure~\ref{f:loopfree}.
\end{defn}

\begin{figure}[h]
\includegraphics{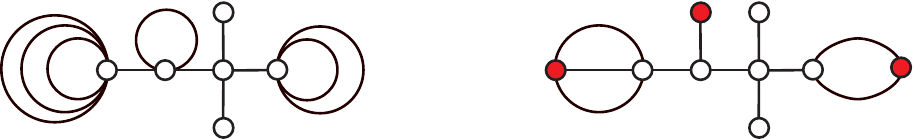}
\caption{A pseudograph $G$ and its associated loop-free version $G_{\LF}$.   The ghost nodes are shaded.}
\label{f:loopfree}
\end{figure}

\begin{prop} \label{p:subcomplex}
For a connected pseudograph $G$ with loops, the graph associahedron $\KG$ can be realized as an open subcomplex of $\KG_{\LF}$. 
\end{prop}

\begin{proof}
The canonical poset inclusion $\phi: \KG \to \KG_{\LF}$ replaces any loop of a tube by its associated edge in $G_{\LF}$. 
This clearly extends to an injection preserving inclusion of tubes, revealing $\KG$ as a subposet of $\KG_\LF$.  Moreover, since covering relations are preserved by $\phi$, $\KG$ is a connected subcomplex of $\K G_{\LF}$.
Indeed, this subcomplex is homeomorphic to a half-space of dimension $n-1+\red$, 
where $\red$ is the number of redundant edges of $G_{\LF}.$  To see this, note the only tubings not in the image of $\phi$ are those containing the singleton ghost tubes.  In $\K G_{\LF}$, those singleton tubes represent  a collection of  pairwise adjacent facets since, by construction, the ghost nodes are never adjacent to each other.  Therefore the image of $\phi$ is a solid polytope minus a union of facets which itself is homeomorphic to a codimension one disk.  
\end{proof}

\begin{cor}
The compact faces of $\KG$ correspond to tubings which exclude all loops.
\end{cor}

\begin{proof}
For any tubing of $T$ in $\KG$ not excluding a loop, $\phi(T)$ will be compatible with the singleton ghost tube in $\KG_\LF$.
\end{proof}

As an added benefit of Theorem~\ref{t:hull} providing a construction of the polytope $\KG_m$, one gets a geometric realization of $\KG$ as a polytopal cone, for pseudographs $G$ with loops.  The result is summarized below, the proof of which is provided at the end of the paper.  Note that in addition to the combinatorial argument, we also see evidence that $\KG$ is conal:  If the removal of one or more hyperplanes creates a larger region with no new vertices, then that region must be unbounded.

\begin{cor} \label{c:loopreal}
The realization of $\KG$ is obtained from the realization of $\KG_\LF$ by removing the halfspaces associated to the singleton tubes of ghost nodes.
\end{cor}

\begin{exmp}
If $G$ is a path with two nodes and one loop, then $G_\LF$ is a path with three nodes.  Figure~\ref{f:2D-reveal}(a) shows the 2D associahedron $\KG_\LF$ from Figure~\ref{f:kwexmp}(a), where the right most node of the path $G_\LF$ can be viewed as a ghost node.  Part (b) shows $\KG$ as seen in Figure~\ref{f:2D-exmp}(b).  Notice that the facet of $\KG_\LF$ corresponding to the tube around the ghost node is removed in (a) to form the open subcomplex of (b).
\end{exmp}

\begin{figure}[h]
\includegraphics{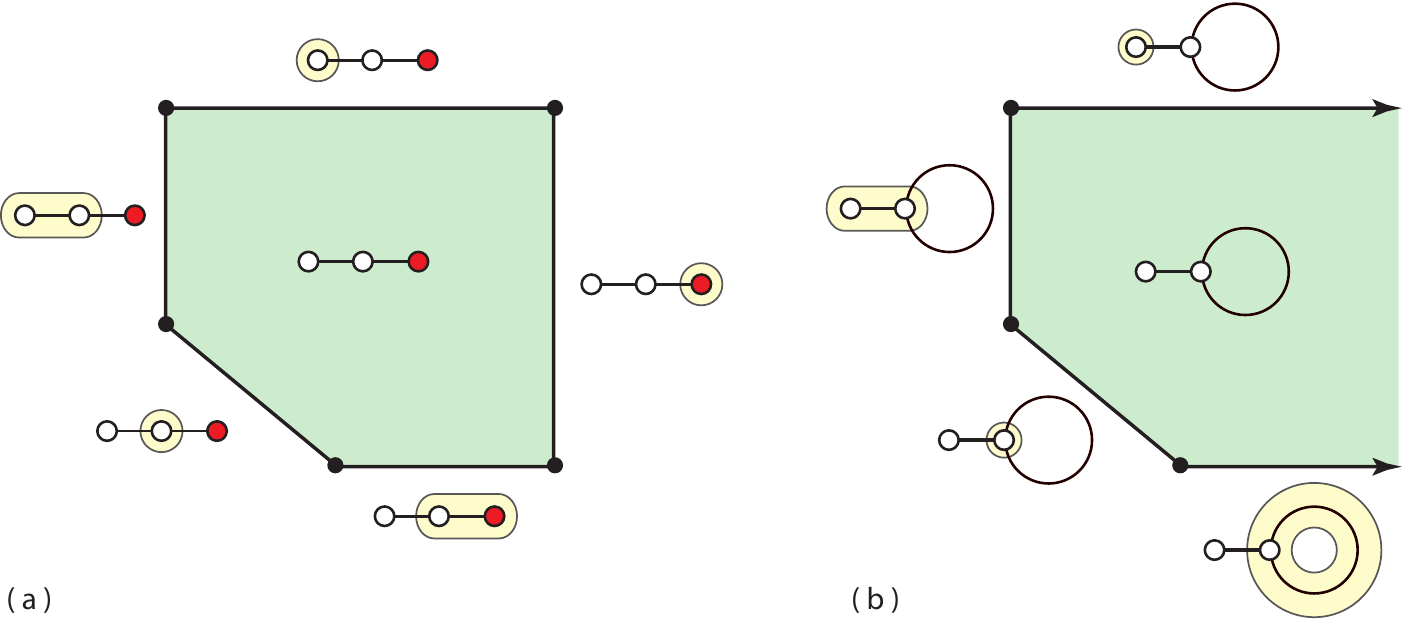}
\caption{(a) The polygon $\KG_\LF$ and (b) the polygonal cone $\KG$.}
\label{f:2D-reveal}
\end{figure}

\begin{exmp}
A 3D version of this phenomena is provided in Figure~\ref{f:3D-loop-reveal}.  Part (a) shows the 3D associahedron, viewed as the loop-free version $\KG_\LF$ to the pseudograph associahedron $\KG$ of part (b).  Indeed, the two labeled facets of (a), associated to tubes around ghost nodes, are removed to construct $\KG$.  The construction of $\KG$ for iterated truncations is given in Figure~\ref{f:3D-loop}.
\end{exmp}

\begin{figure}[h]
\includegraphics{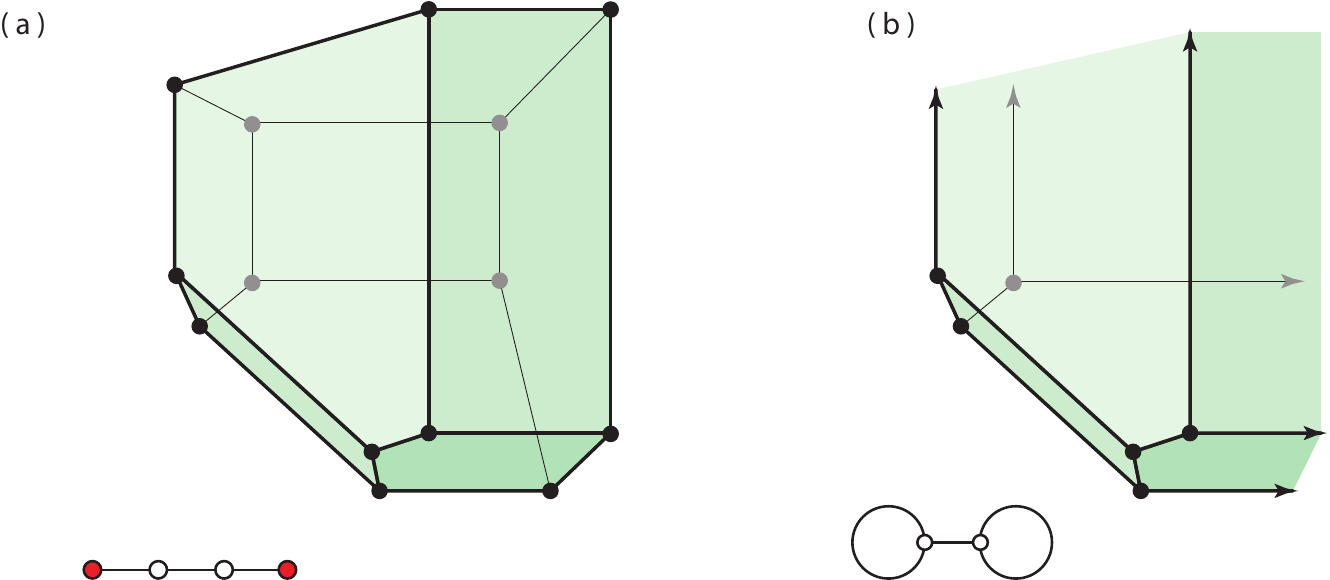}
\caption{(a) The associahedron $\KG_\LF$ and the (b) polyhedral cone $\KG$, where the faces of $\KG_\LF$ associated to tubes around ghost nodes have been removed.}
\label{f:3D-loop-reveal}
\end{figure}

\begin{exmp}
A similar situation can be seen in Figure~\ref{f:3D-permuto-flat}, part (a) showing the permutohedral prism $\KG_\LF$ and part (b) the cone $\KG$ after removing the back face of the prism.
\end{exmp}

%
%
\section{Proofs}  \label{s:proof}
\subsection{}

The proof of Theorem~\ref{t:trunc} is now given, which immediately gives a proof of Theorem~\ref{t:pseudo}.
We begin with a description of the structure of $\tbase_G$, the polytope given in~\eqref{e:midtrunc}.
\begin{enumerate}
\item Each face corresponds to a tubing consisting of full tubes $T^F$ and to a subset $S$ of the edges and loops of $G$.   The set $S$ contains at least one edge of each bundle. 
\item The subset $S$ produces a tube $G_S$ that contains all the nodes of $G$ as well as $S$.
\item The tubing for a given face is the \emph{intersection} of $T^F$ with $G_S$.
\item A face with a tubing $T_a$ contains a face with a tubing $T_b$, if and only if $T_a^F \subset T_b^F$ and $G_{S_a} \supset G_{s_b}$.
\item Given two faces with tubings $T_a$ and $T_b$, their intersection is the intersection of $T_a^F \cup T_b^F$ with $G_{S_a} \cap G_{S_b}$ assuming the former is a tubing and the latter is a tube.  Otherwise the faces do not intersect.
\end{enumerate}
In order to describe the effect of truncation on these tubings, we define  \emph{promotion}, an operation on sets of tubings that was developed in \cite[Section 2]{cd}.

\begin{defn}
The \emph{promotion} of a tube $G_t$ in a set of tubings $\tubingset{}$ means adding to $\tubingset{}$ the tubings 
$$\{T \cup \{G_t\} \suchthat T \in \tubingset{}, G_t \text{ is compatible with all } G_{t'} \in T \} \, .$$
Note that this $T$ may be empty.  The new tubings are ordered such that $T \cup \{G_t\} \prec T$, and $T \cup \{G_t\} \prec T' \cup \{G_t\}$ if and only if $T \prec T'$ in $\tubingset{}$.
\end{defn}

All valid combinations of full tubes of $G$ already exist as faces of $\tbase_G$.  They are also already ordered by containment.  Therefore, we may first conclude from this definition that promoting the non-full tubes is sufficient to produce the set of all valid tubings of $G$, resulting in $\KG$.
Given a polytope whose faces correspond to a set of tubings, promoting a tube $G_F$ is equivalent to truncating its corresponding face $F$ so long as the subset of tubings compatible with $G_F$ corresponds to the set of faces that properly intersect or contain $F$.  Verifying this equivalence for each prescribed truncation is sufficient to prove the theorem.  

\begin{proof}[Proof of Theorem~\ref{t:trunc}]
We may proceed by induction, relying on the description of $\tbase_G$ above and leaving the computations of intersections to the reader.  Consider the polytope $P$ in which all the faces before $F$ in the prescribed order have been truncated.   Suppose that until this point, the promotions and truncations have been equivalent, that is, there is a poset isomorphism between the base polytope after a set of truncations and the sets of base tubings after the set of corresponding tubes are promoted.  Note that in $P$, the faces that intersect (but are not contained in) $F$ are 
\begin{enumerate}
\item faces that properly intersected or contained $F$ in $\tbase_G$
\item faces corresponding to tubes promoted before $G_F$ and compatible with $G_F$. 
\end{enumerate}
Since faces created by truncation inherit intersection data from both the truncated face and the intersecting face, we may include (by induction if necessary) any intersection of the above that exists in $P$.  
Conversely, the faces that do not intersect $F$ in $P$ are
\begin{enumerate}
\item faces that did not intersect $F$ in $\tbase_G$
\item faces that did intersect $F$ but whose intersection was contained in a face truncated before $F$ and was thus removed
\item faces corresponding to tubes promoted before $G_F$ but incompatible with $G_F$
\item any intersection of the above that exists in $P$.
\end{enumerate}
We have given a description of when no intersection exists between two faces in $\tbase_G$, as case (1) above.  Most tubings incompatible with $G_F$ can be shown to belong to such a group. 
Some tubes $G_t$ that intersect $G_F$ fall into case (2), where their intersection corresponds to $\{G_t , G_t \cap G_F\}$.  It is contained in the face corresponding to $\{G_F \cap G_t\}$, a face found before $G_F$ in the containment order.  Thus no intersection is present in $P$.

The tubings compatible with $G_F$ correspond to the faces that properly intersect or contain $F$. Promoting $G_F$ and truncating $F$ will produce isomorphic face/tubing sets.  The conclusion of the induction is that the prescribed truncations will produce a polytope isomorphic to the set of tubings of $G$ after all non-full tubes have been promoted, resulting in $\KG$.
\end{proof}

\subsection{}

We now provide the proof for Theorem~\ref{t:hull}.  As before, let $G$ be a pseudograph without loops, and let $T$ be a maximal tubing of $G$. Moreover, let $\CH$ denote the polytope obtained from the convex hull of the points in Equation~\eqref{e:hull}.  Close inspection reveals that $\CH$ is contained in an intersection of the hyperplanes defined by the equations:
\begin{align*}
h_V \ : \ \ & \sum_{v \in V} f_T(v) \ =\ c^{|V|}\\
h_{B_i} \ : \ \ & \sum_{e \in B_i} f_T(e) \ = \ c^{b_i}
\end{align*}
where $|V|$ is the number of nodes of $G$.
To each tube $G_t \in T$, let
$$\Lambda(G_t) \ = \ c^{\; |V(t)|} \ + \ \sum_i c^{\; E(i,t)} \ + \ \left|G - G_t\right|^2 \, .$$
These $\Lambda$ functions define halfspaces which contain the vertices associated to that tube:
$$h_t^+ \ : \ \ \sum_{x \in G_t} f_T(x) \ \geq \ \Lambda(G_t)\, .$$  
Proving that $\CH$ has the correct face poset as $\KG$ is mostly a matter of showing the equivalence of $\CH$ and the region
$$\HY \ := \ h_V \ \cap \ \bigcap_i h_{B_i} \ \cap \ \bigcap_{G_t \in T} h_t^+ \, .$$

\begin{defn}
Two tubes $G_a$ and $G_b$ of $G$ are \emph{bundle compatible} if  for each $i$, one of the sets $E(i,a)$ and $E(i,b)$ contains the other.  Note that the tubes of any tubing $T$ are pairwise (possibly trivially) bundle compatible.
\end{defn}

\begin{lem} \label{l:inequality}
Let $G_{a}$ and $G_{b}$ be adjacent or properly intersecting bundle compatible tubes. Suppose their intersection is a set of tubes $\{G_{\wedge_i}\}$, while $G_{\vee}$ is a minimal tube that contains both.  Let $E_\vee$ be the set of edges contained in $G_{\vee}$ but not $G_{a}$ or $G_{b}$.  Then for any tubing $T$ containing $G_{\vee}$,  
$$\Lambda(G_{a}) \ < \ \Lambda(G_{\vee}) \ - \ \Lambda(G_{b}) \ + \ \sum_i \Lambda(G_{\wedge_i}) \ - \ \sum_{e \in E_\vee} f_T(e).$$
\end{lem}

\begin{proof}
The intersections with each bundle contribute equally to both sides.  If $G_{\vee}$ contains more nodes than the others, then we simply note the dominance of the $k^{|V(\vee)|}$ term and place bounds on the remaining ones.  If not, the sides are identical up to the $|G-G_t|^2$ terms, which provide the inequality.
\end{proof}

\begin{lem} \label{l:halfspace}
For any tubing $T$, and any tube $G_t$, 
\begin{equation}
\label{e:halfspace}
\sum_{x \in G_t} f_T(x) \ \geq \ \Lambda(G_t)
\end{equation}
with equality if and only if $G_t \in T$.
In particular, $\CH \subseteq \HY$, and only those vertices of $\CH$ that have $G_t$ in their tubing are contained in $h_t$.
\end{lem}


\begin{proof}
If $G_t \in T$, the equality of Equation~\eqref{e:halfspace} follows directly from the definition of $f_T$.
Suppose then that $G_t \notin T$.  We proceed by induction on the size of $G_t$.
First, produce a tube $G_{\sigma}$ which contains the same nodes as $G_t$, and the same size intersection with each bundle, but is bundle compatible with the tubes of $T$.  Naturally $\Lambda(G_{\sigma})=\Lambda(G_t)$, but since $f_T$ is an increasing function over the ordered $e(i,j)$ edges of $G$, we get
$$\sum_{x \in G_t} f_T(x) \ \geq \ \sum_{x \in G_\sigma} f_T(x)$$ with equality only if $G_t=G_{\sigma}$.

Let $G_\vee$ be the smallest tube of $T$ that contains $G_\sigma$ (or all of G if none exists).  If $G_\vee = G_\sigma$ then the inequality above is strict and the lemma is proven.  Otherwise the maximal subtubes $\{G_{\vee_i}\}$ of $G_\vee$ are disjoint, and each either intersects or is adjacent to $G_\sigma$.  If we denote the intersections as $\{G_{\wedge_i}\}$ and the set of edges of $G_\vee$ contained in none of these subtubes by $E_\vee$, then as a set, 
$$G_\sigma \ = \ G_\vee \ - \ \bigcup_i G_{\vee_i} \ + \ \bigcup_i G_{\wedge_i} \ - \ \bigcup_{e \in E_\vee} f_T(e) \, .$$
The tubes mentioned in the right hand side are all in $T$, except perhaps the intersections.  Fortunately, the inductive hypothesis indicates that  
$$\sum_{x \in G_{\wedge_i}} f_T(x) \ \geq \ \Lambda( G_{\wedge_i})\, .$$ 
Thus we are able to rewrite and conclude
$$\sum_{x \in G_\sigma} f_T(x) \ \geq \ \Lambda (G_\vee) \ - \ \sum_i \Lambda(G_i) \ + \ \sum_i \Lambda (G_{\wedge_i}) \ - \ \sum f_T(e_i) \ > \ \Lambda(G_t)$$
by repeated applications of Lemma \ref{l:inequality}.
\end{proof}

\begin{lem} \label{l:halfspaceinpoly}
$\HY \subseteq \CH$.
\end{lem}

\begin{proof}
Particular half spaces impose especially useful bounds of the value of certain coordinates within $\HY$.  For instance, if $G_w$ is a full tube, then 
$$h_w^+ \ : \ \ \sum_{v \in V(w)} f_T(v) \ \geq \ c^{|V(w)|} + |G-G_w|^2 \, .$$
Choosing the maximal tube $G_x$ that intersects bundle $B_i$ in a particular subset of edges $X$ produces 
$$h_x^+ \ : \ \ \sum_{e \in X} f_T(e) \ \geq \ c^{|X|} + |G-G_x|^2 \, .$$
Applying these to single nodes and single edges gives a lower bound in each coordinate.  The hyperplanes $h_V$ and $h_{B_i}$ supply upper bounds, so $\HY$ is bounded.

Suppose $\HY - \CH$ is not empty.  Since $\CH$ is convex, by construction, $\HY - \CH$ must have a vertex $v^*$ outside $\CH$, at the intersection of several $h_t$ hyperplanes.  These hyperplanes correspond to a set $T^*$ of tubes of $G$.  This $T^*$ contains at least one pair of incompatible tubes $G_a$ and $G_b$, for otherwise it would be a tubing and $v^*$ would be in $\CH$.

\begin{enumerate}
\item
If $G_a$ and $G_b$ are bundle incompatible
in some bundle $B_i$, then we produce the maximal tube $G_u$ that intersects $B_i$ in $E(i,a) \cup E(i,b)$.  As above, $G_u$ produces a bound on the $E(i,u)$ coordinates, yielding
$$h_u^+ \ : \ \ \sum_{e \in E(i,u)} f_T(e) \ \geq \ c^{|E(i,u)|} + |G-G_u|^2 \, .$$ 
The half spaces $h_w^+$ and $h_x^+$ above produce lower bounds on the sum of the vertex coordinates of $G_a$ and $G_b$. 
Subtracting these from $\Lambda (G_a)$ and $\Lambda(G_b)$ leaves a maximum of 
$$c^{|E(i,a)|} \ + \ \left|G-G_a\right|^2 \ + \ c^{|E(i,b)|} \ + \left|G-G_b\right|^2$$
for $\sum_{E(i,a)} f_T(e)$ and $\sum_{E(i,b)} f_T(e)$, which is insufficient for the $G_u$ requirement above.  We conclude that $v^*$  is either outside $h_u^+$ or outside one of the halfspaces $h_w^+$ or $h_x^+$.  Either way, $v^*$ is not in $\HY$.  

\item
On the other hand, if $G_a$ and $G_b$ are bundle compatible, Lemma \ref{l:inequality} can be rearranged:
$$\Lambda(G_\vee) \ > \ \Lambda(G_a) \ + \ \Lambda(G_b) \  - \ \sum_i \Lambda(G_{\wedge_i}) \ + \ \sum_{e \in E_{\vee}} f_T(e) \, .$$
Thus $v^*$ is either not in one of the $h_{\wedge_i}^+$ or not in $h_\vee^+$.
Therefore $v^*$ is not in $\HY$. 
\end{enumerate}
This contradiction proves the Lemma.
\end{proof}

\begin{proof}[Proof of Theorem~\ref{t:hull}]
Lemmas~\ref{l:halfspace} and~\ref{l:halfspaceinpoly} show that $\CH = \HY$.
Consider the map taking a tubing $T$ of $G$ to the face
$$\CH \cap \bigcap_{G_t \in T} h_t$$ of $\CH$.
By Lemma~\ref{l:halfspace}, each tubing maps to a face of $\CH$ containing a unique set of vertices. Each face is an intersection of hyperplanes that contains such a vertex (and hence corresponds to a subset of a valid tubing).  Since it clearly reverses containment, this map is an order preserving bijection.
\end{proof}

\begin{proof}[Proof of Corollary~\ref{c:loopreal}]
We remark that notation (and the entire reasoning) in this proof is being imported from the proof of Lemma~\ref{l:halfspaceinpoly}.
If  $v$ is a ghost node, then it is not $G_w$, $G_x$ or $G_u$ for a pair of bundle incompatible tubes (since those tubes all have at least 2 nodes).  It also is neither $G_\vee$ nor $G_{\wedge_i}$ for any pair of bundle compatible tubes. Thus $h_t^+$ excludes no intersection of hyperplanes. Its removal creates no new faces, and removes only those faces corresponding to tubings containing $v$.  The identification of these faces is the canonical poset inclusion $\phi$ from the proof of Proposition~\ref{p:subcomplex}.
\end{proof}

%
%
\bibliographystyle{amsplain}

\begin{thebibliography}{XX}

\baselineskip=15pt


\bibitem[1]{arw} F.\ Ardila, V.\ Reiner, L.\ Williams. Bergman complexes, Coxeter arrangements, and graph associahedra, \emph{Seminaire Lotharingien de Combinatoire} {\bf 54A}(2006).

\bibitem[2]{blo} J.\ Bloom.  A link surgery spectral sequence in monopole Floer homology, preprint arxiv:0909.0816.



\bibitem[3]{cd} M.\ Carr and S.\ Devadoss. Coxeter complexes and graph-associahedra, \emph{Topology and its Applications} {\bf 153} (2006) 2155-2168.

\bibitem[4]{djs} M.\ Davis, T.\ Januszkiewicz, R.\ Scott.  Fundamental groups of blow-ups, \emph{Advances in Mathematics} {\bf 177} (2003) 115-179.

\bibitem[5]{dp} C.\ De Concini and C.\ Procesi. Wonderful models of subspace arrangements, {\em Selecta Mathematica} {\bf 1} (1995), 459-494.

\bibitem[6]{dev1} S.\ Devadoss. Tessellations of moduli spaces and the mosaic operad, in \emph{Homotopy Invariant Algebraic Structures}, Contemporary Mathematics {\bf 239} (1999) 91-114

\bibitem[7]{dev2} S.\ Devadoss. A realization of graph-associahedra, \emph{Discrete Mathematics} {\bf 309} (2009) 271-276.



\bibitem[8]{fs} S.\ Forcey and D.\ Springfield.  Geometric combinatorial algebras: cyclohedron and simplex,
\emph{Journal of Algebraic Combinatorics}, to appear.



\bibitem[9]{lr} J.-L.\ Loday and M.\ Ronco. Hopf algebra of the planar binary trees, \emph{Advances in Mathematics} {\bf 139} (1998) 293--309.

\bibitem[10]{mps}J.\ Morton, L.\ Pachter, A.\ Shiu, B.\ Sturmfels, O.\ Wienand.  Convex rank tests and semigraphoids, \emph{SIAM Journal on Discrete Mathematics}, to appear.

\bibitem[11]{prw} A.\ Postnikov, V.\ Reiner, L.\ Williams.  Faces of generalized permutohedra, \emph{Documenta Mathematica}, to appear.


\bibitem[12]{to} A.\ Tonks.  Relating the associahedron and the permutohedron, in \emph{Operads: Proceedings of Renaissance Conference}, Contemporary Mathematics {\bf 202} (1997) 33-36.


\end{thebibliography}

\end{document}